\documentclass{scrartcl}
\usepackage[
color, 
noindent, 
english, 
enumstyle, 
theorem, 
operator, 
alphabet, 
cmd, 
arraytweak,
numberbysection,
arraytweak
]
{jgrygierek}
\usepackage[draft=false]{scrlayer-scrpage}
\usepackage[normalem]{ulem}
\usepackage[
backend=biber, 
maxnames=42,
style=alphabetic,
citestyle=alphabetic,
doi=false,
url=false,
isbn=false,
giveninits=true
]{biblatex} % because of 42 reasons ;-)
\patchcmd{\thebibliography}{\chapter*}{\section*}{}{}
\allowdisplaybreaks
\date{\today}
\author{\texorpdfstring{Jens Grygierek\footnotemark[1]}{Jens Grygierek}}
\title{Poisson and Gaussian fluctuations for the \texorpdfstring{$\bff$}{f}-vector of high-dimensional random simplicial complexes} % no // to keep hyperref happy ;)
\makeatletter
\hypersetup{
  pdftitle    = {\@title},
  pdfsubject  = {Mathematics/Probability, 60D05, 60F05, 55U10},
  pdfauthor   = {\@author},
  pdfkeywords = {\@title, \@author, Poisson limit theorem, central limit theorem, high dimensional random Vietoris-Rips complex, Poisson point process, second-order Poincar\'{e} inequality, stochastic geometry, phase transition}
}
\makeatother

\begin{document}
\renewcommand{\thefootnote}{\fnsymbol{footnote}}

\maketitle
	
\footnotetext[1]{Institute of Mathematics, Osnabr\"uck University, Germany. Email: jens.grygierek@uni-osnabrueck.de}
	
%\nocite{*}
\begin{abstract}	
	\noindent We investigate the high-dimensional asymptotic distributional behavior of the $\bff$-vector of a random Vietoris-Rips complex, that is generated over a stationary Poisson point process in $[-\frac{1}{2}, \frac{1}{2}]^d$ as the space dimension and the intensity tend to infinity while the radius parameter tends to zero simultaneously.
	% http://www.ams.org/msc/msc2010.html
\end{abstract}

\section{Introduction}\label{sec:intro}

The field of topological data analysis motivates the study of random simplicial complexes, especially random geometric complexes that are higher-dimensional generalizations of the well known random geometric graph.
Naturally, one builds a simplicial complex on data points to study features of the data using combinatorial or topological properties like the $\bff$-vector, that counts the number of $k$-dimensional simplices, the Betti-numbers or persistent homology.
For a recent introduction into the different opportunities in this research field, we refer to the survey article \cite{BobrowskiKahle2018}.

Let $\eta_d$ be a stationary Poisson point process in $W := [-\frac{1}{2},+\frac{1}{2}]^d \subset \RR^d$ with dimension dependent intensity $t_d \in (0,\infty)$, i.e. the intensity measure is given by $\mu_d = t_d \Lambda_d$, where $\Lambda_d$ denotes the $d$-dimensional Lebesgue measure.
We choose a dimension-dependent distance parameter $\delta_d \in (0,\tfrac{1}{4})$ with $\delta_d \rightarrow 0$ for $d \rightarrow \infty$.

The points of $\eta_d$ are taken as the vertices of the random Vietoris-Rips complex $\vietoris^{\infty}(\eta_d, \delta_d)$, that contains any $k$-dimensional simplex $\cbras*{x_0,\ldots,x_k} \subseteq {\eta_d}^{k+1}_{\neq}$, $k \in \NNN$, if and only if the pairwise uniform distances of its vertices are bounded by $\delta_d$, i.e.
\begin{align*}
	\cbras*{x_0,\ldots,x_k} \in \vietoris^{\infty}(\eta_d,\delta_d) :\Leftrightarrow \unorm{x_j - x_i} \leq \delta_d \quad \text{for all} \quad i,j \in \cbras*{0,\ldots,k}.
\end{align*}
The collection of all $1$-dimensional simplices coincides with the edges of the well known random geometric graph, where the points of $\eta_d$ are taken as the vertices and any two vertices are connected by an edge whenever their uniform distance is less than or equal to $\delta_d$, see \cite{Penrose2003} for more details.

To simplify our notation, we will mostly omit the index $d$ in the following. Nevertheless all conditions we impose on the parameter sequences $t := (t_d)_d$ and $\delta := (\delta_d)_d$ in the following have to be treated with respect to $d \rightarrow \infty$.

Let $F_k := F_k(\vietoris^{\infty}(\eta_d,\delta_d))$, $k \geq 1$, denote the number of $k$-simplices in the random Vietoris-Rips complex $\vietoris^{\infty}(\eta_d,\delta_d)$, that is the $U$-statistic of order $k+1$ given by
\begin{align*}
	F_k(\vietoris^{\infty}(\eta_d,\delta_d)) := \frac{1}{(k+1)!} \sum\limits_{(y_0,\ldots,y_k) \in \eta^{k+1}_{\neq}} \prod\limits_{i=0}^k \prod\limits_{j=i+1}^k \1\cbras*{\unorm*{y_j - y_i} \leq \delta},
\end{align*}
where $\unorm*{\cdot}$ denotes the uniform norm on $\RR^d$. 
Note that $F_k$ is the $k$-th component of the $\bff$-vector of $\vietoris^{\infty}(\eta_d,\delta_d)$, i.e.
\begin{align*}
	F_k = f_k(\vietoris^{\infty}(\eta_d,\delta_d)).
\end{align*}
Additionally, $F_k$ counts the complete sub-graphs with $k+1$ vertices in the random geometric graph with respect to the uniform distance $\dist_{\infty}(x,y) = \unorm*{x-y}$.

We investigate the asymptotic distributional behavior of $F_k$ as $\delta \rightarrow 0$ and the intensity as well as the space dimension $d$ tend to infinity simultaneously.

\subsection{Main results}

As a preparation for our limit theorems we show asymptotically sharp bounds for the expectation and the variance of our $k$-simplex counting functional:

\begin{lemma}\label{lem:expectation}
	For all $d \geq 1$ and $k \geq 1$ the expected number of $k$-simplices in the random Vietoris-Rips complex $\vietoris^{\infty}(\eta_d,\delta_d)$ is bounded by
	\begin{align*}
		\rbras*{1-2\delta}^d \frac{t\rbras*{t\delta^d}^k (k+1)^d}{(k+1)!} \leq \E*{F_k} \leq \frac{t\rbras*{t\delta^d}^k (k+1)^d}{(k+1)!}.
	\end{align*}
\end{lemma}

\begin{lemma}\label{lem:variance}
	For all $d \geq 1$ and $k \geq 1$ there existing explicit constants $\bC(k,r) \in (0,\infty)$ only depending on $k$ and $r$ such that the variance of the number of $k$-simplices in the random Vietoris-Rips complex $\vietoris^{\infty}(\eta_d,\delta_d)$ is bounded by
	\begin{align*}
		\V*{F_k} & \geq \E*{F_k} + \rbras*{1-2\delta}^d t(t\delta^d)^k \sum\limits_{r=1}^k \bC(k,r) (t\delta^d)^{k+1-r} \rbras*{\tfrac{2(k+2)(k+1-r)}{r+1} + r}^d,\\
		\V*{F_k} & \leq \E*{F_k} + t(t\delta^d)^k \sum\limits_{r=1}^k \bC(k,r) (t\delta^d)^{k+1-r} \rbras*{\tfrac{2(k+2)(k+1-r)}{r+1} + r}^d.
	\end{align*}
\end{lemma}

To ensure that the lower and upper bound for the expectation and variance tend to the same limit, we assume that $(\delta_d)_{d \in \NNN}$ is decreasing sufficiently fast, i.e. we assume
\begin{align*}
	\lim\limits_{d \rightarrow \infty} d \delta_d = 0,
\end{align*}
see Remark \ref{rem:delta:convergence} for more details.

The asymptotic behavior of $F_k$ depends on how fast the sequence $(t_d)_{d \in \NNN}$ increases as $d \rightarrow \infty$.
This phenomenon is quite common for asymptotic results related to edge counts in fixed dimension and was also shown for edge-counts in high-dimensional random geometric graphs in our previous work \cite{GrygierekThaele2016, Grygierek2019b} considering a slightly different model.

In particular, here, one has to distinguish the following phases, determined by the limit of the expectation $\E*{F_k}$:
\begin{align}
	\label{eq:phase:1}
	\lim\limits_{d \rightarrow \infty} \frac{1}{(k+1)!} t(t\delta^d)^k(k+1)^d & = \infty,\\
	\label{eq:phase:2}
	\lim\limits_{d \rightarrow \infty} \frac{1}{(k+1)!} t(t\delta^d)^k(k+1)^d & = \theta \in (0,\infty),\\
	\label{eq:phase:3}
	\lim\limits_{d \rightarrow \infty} \frac{1}{(k+1)!} t(t\delta^d)^k(k+1)^d & = 0.
\end{align}

The rate of convergence in the following central limit theorem and Poisson limit theorem will be measured by the so-called Wasserstein distance $\dist_W(\cdot, \cdot)$ resp. the total variation distance $\dist_{TV}(\cdot,\cdot)$, see Section \ref{sec:prelim:malliavin} below for a formal definition. We indicate convergence in distribution by writing $\os{D}{\rightarrow}$.

If the expectation tends to infinity \eqref{eq:phase:1} the $k$-simplex counting functional satisfies a central limit theorem:

\begin{theorem}[Gaussian Approximation]\label{thm:gaussian-limit}
	For $k \geq 1$ fixed, we assume $\E*{F_k} \rightarrow \infty $ for $d \rightarrow \infty$. 
	Let $\cN(0,1)$ be a standard Gaussian distributed random variable and denote by $\widetilde{F_k} := \tfrac{F_k - \E*{F_k}}{\sqrt{\V*{F_k}}}$ the standardized version of $F_k$.
	
	If $(t\delta^d) \rightarrow 0$ for $d \rightarrow \infty$, then
	\begin{align*}
		\dist_{W}(\widetilde{F_k}, \cN(0,1))
%		& =
%		\begin{cases}
%			\cO\rbras*{t^{-\frac{1}{2}} (t\delta^d)^{-\frac{k}{2}} 2^d(k+1)^d}, & k \leq 3\\
%			\cO\rbras*{t^{-\frac{1}{2}} (t\delta^d)^{-\frac{k}{2}} (k+1)^{\frac{3d}{2}}}, & k \geq 3
%		\end{cases}
		& =
		\begin{cases}
			\cO\rbras*{\rbras*{\E*{F_k}}^{-\frac{1}{2}} (k+1)^{\frac{3d}{2}}2^d}, & k \leq 3,\\
			\cO\rbras*{\rbras*{\E*{F_k}}^{-\frac{1}{2}} (k+1)^{2d}}, & k \geq 3.
		\end{cases}
	\end{align*}
	If $(t\delta^d) \rightarrow c \in (0,\infty)$ or $(t\delta^d) \rightarrow \infty$ for $d \rightarrow \infty$, then
	\begin{align*}
		\dist_{W}(\widetilde{F_k}, \cN(0,1))
		=
		\cO\rbras*{t^{-\frac{1}{2}} \rbras*{1+\tfrac{1}{k^2+2k}}^d 2^d}.
	\end{align*}
	In particular, if $\E*{F_k}$ resp. $t$ is increasing sufficiently fast depending on $d$ one has that
	\begin{align*}
		\widetilde{F_k} \os{D}{\rightarrow} \cN(0,1), \quad \text{as} \quad d \rightarrow \infty.
	\end{align*}
\end{theorem}

If the expectation tends to a finite positive limit \eqref{eq:phase:2} the $k$-simplex counting functional satisfies a Poisson limit theorem:

\begin{theorem}[Poisson Approximation]\label{thm:poisson-limit}
	For $k \geq 1$ fixed, we assume $\E*{F_k} \rightarrow \theta \in (0,\infty)$ for $d \rightarrow \infty$. 
	Let $\cP(\theta)$ be a Poisson distributed random variable with expectation and variance $\theta$. Then 
	\begin{align*}
		\dist_{TV}(F_k,\cP(\theta)) = \cO\rbras*{\abs*{\E*{F_k} - \theta}} + \cO\rbras*{\abs*{\V*{F_k} - \theta}} + 
		\begin{cases}
			\cO\rbras*{t^{-\frac{1}{2k}} (k+1)^{\frac{d(3k-1)}{2k}} 2^{d}}, & k \leq 3,\\
			\cO\rbras*{t^{-\frac{1}{2k}} (k+1)^{\frac{d(4k-1)}{2k}}}, & k \geq 3.
		\end{cases}
	\end{align*}
	In particular, if $t$ is increasing sufficiently fast depending on $d$ one has that
	\begin{align*}
		F_k \os{D}{\rightarrow} \cP(\theta), \quad \text{as} \quad d \rightarrow \infty.
	\end{align*}
\end{theorem}

\begin{remark}
	If the expectation tends to zero \eqref{eq:phase:3} we also have $\V*{F_k} \rightarrow 0$, indicating that the $k$-simplex counting functional vanishes in the limit, since the random Vietoris-Rips complex contains almost surely no $k$-simplices.
\end{remark}

\begin{conjecture}
	The dimension dependent factors in the bounds on the probability-metrics in the limit theorems presented above impose a condition on the growth of the intensity $t$ as $d$ tends to infinity.
	We may conjecture, that this condition can be weakened or even dropped completely, by improving the dimension depend factors if the exact values of the integrals arising in Theorem \ref{thm:derivative:bound} are calculated and used similar to the proof of Lemma \ref{lem:variance:limit}.
\end{conjecture}

\bigskip

This paper is organized as follows. 
For the convenience of the reader, we repeat the relevant material on the Malliavin-Stein method for normal approximation and Poisson approximation of Poisson functionals in Section \ref{sec:prelim:malliavin}. Additionally, we give a short introduction to simplicial complexes in Section \ref{sec:prelim:simcomp}.
In Section \ref{sec:decomp} we introduce a decomposition technique for $U$-statistics that will be used in the proof of our main results, that are presented in Section \ref{sec:proof}:
We start with the expectation and variance bounds, Lemmas \ref{lem:expectation} and \ref{lem:variance} in Subsection \ref{sec:proof:exp-var}.
In Subsection \ref{sec:proof:malliavin-derivative} we prepare estimations on the first and second order Malliavin derivatives, that will finally be used in Subsection \ref{sec:proof:limits} to obtain the central limit theorems, Theorem \ref{thm:gaussian-limit}, and the Poisson limit theorem, Theorem \ref{thm:poisson-limit}.

\section{Preliminaries}\label{sec:prelim}

The $d$-dimensional Euclidean space is denoted by $\RR^d$ and we let $\sB^d$ be the Borel $\sigma$-field on $\RR^d$.
The Lebesgue measure on $\RR^d$ is indicated by $\Lambda_d$. The $d$-dimensional closed ball with respect to the uniform norm, center in $z \in \RR^d$ and radius $r > 0$ is defined by
\begin{align*}
	\uBB zr := \cbras*{x \in \RR^d : \unorm{x-z} \leq r}.
\end{align*}

\subsection{Poisson functionals and difference operators}

Let $\rN_\sigma$ denote the class of all $\sigma$-finite counting measures $\chi$ on $\RR^d$, i.e. $\chi(B) \in \NNN \cup \{ \infty \}$ for all $B \in \sB^d$. We equip the space $\rN_\sigma$ with the $\sigma$-field $\sN_\sigma$ generated by the mappings $\chi \rightarrow \chi(B)$, $B \in \sB^d$.

\begin{definition}
	A Poisson point process with intensity measure $\mu$ is a random counting measure on $\RR^d$, that is a random element in the space $\rN_\sigma$, that satisfies the following properties:
	\begin{enumerate}
		\item For all $B \in \sB^d$ and all $k \in \NNN$ it holds that $\eta(B)$ is a Poisson distributed random variable with expectation $\mu(B)$, i.e.
			\begin{align*}
				\Prob*{\eta(B) = k} = \frac{\mu(B)^k}{k!}e^{-\mu(B)},
			\end{align*}
			where we set $\tfrac{\infty^k}{k!}e^{-\infty} = 0$ for all $k$ if $\mu(B) = \infty$.
		\item For all $m \in \NNN$ and all pairwise disjoint measurable sets $B_0, \ldots, B_m \in \sB^d$, the random variables $\eta(B_0), \ldots, \eta(B_m)$ are independent.
	\end{enumerate}
\end{definition}

To simplify our notation we will often handle $\eta$ as a random set of points using
\begin{align*}
	x \in \eta \Leftrightarrow x \in \cbras*{y \in \RR^d : \eta(\cbras{y}) > 0}.
\end{align*}

It is well known that such a Poisson point process $\eta$ satisfies the following multivariate Mecke formula, see \cite[Theorem 4.4]{LastPenrose2018}.

\begin{lemma}
	For all $m \in \NN$ and all non-negative measurable functions $h:(\RR^d)^m \times \rN_\sigma \rightarrow \RR$ it holds, that
	\begin{align}\label{eq:mecke}
		\begin{split}
		& \dsE \sum\limits_{(y_1,\ldots,y_m) \in \eta^m_{\neq}} h(y_1,\ldots,y_m; \eta)\\
		& \quad = \int\limits_{(\RR^d)^m} \E*{h(y_1,\ldots,y_m;\eta + \delta_{y_1} + \ldots + \delta_{y_m})} \id \mu^m(y_1,\ldots,y_m),
		\end{split}
	\end{align}
	where $\eta^m_{\neq}$ is the collection of $m$-tuples of pairwise distinct points charged by $\eta$.
\end{lemma}

We call a random variable $F$ a Poisson functional if there exists a measurable map $f:\rN_\sigma \rightarrow \RR$ such that $F = f(\eta)$ almost surely.
The map $f$ is called the representative of $F$.
We define the difference operator or so called ``add-one-cost operator'':

\begin{definition}\label{def:malliavin-difference-operator}
	Let $F$ be a Poisson functional and $f$ its corresponding representative, then the first order difference operator is defined by
	\begin{align*}
		D_xF := f(\eta + \delta_x) - f(\eta), \quad x \in \RR^d,
	\end{align*}
	where $\delta_x$ denotes the Dirac measure with mass concentrated in $x$.
	We say that $F$ belongs to the domain of the difference operator, i.e. $F \in \dom{D}$, if $\E*{F^2} < \infty$ and
	\begin{align*}
		\int\limits_{\RR^d} \! \E*{(D_xF)^2} \, \mu(\id x) < \infty.
	\end{align*}
	The second order difference operator is obtained through iteration:
	\begin{align*}
		D_{x_1,x_2}^2F & := D_{x_1}(D_{x_2}F)\\
			& \phantom{:}= f(\eta + \delta_{x_1} + \delta_{x_2}) - f(\eta + \delta_{x_1}) - f(\eta + \delta_{x_2}) + f(\eta), \quad x_1, x_2 \in \RR^d.
	\end{align*}
\end{definition}

For a deeper discussion of the underlying theory of Poisson point processes, Malliavin-Calculus, the Wiener-It\^{o} chaos expansion and the Malliavin-Stein method presented below, see \cite{PeccatiReitzner2016} and \cite{LastPenrose2018}.

\subsection{Malliavin-Stein method}\label{sec:prelim:malliavin}

We will use the Wasserstein-distance for the normal approximation and the total variation distance for the Poisson approximation, see for instance \cite[Section 2.1]{BourguinPeccati2016}.

\begin{definition}\label{def:wasserstein}
	We denote by $\Lip{1}$ the class of Lipschitz functions $h:\RR \rightarrow \RR$ with Lipschitz constant less or equal to one, i.e. $h$ is absolutely continuous and almost everywhere differentiable with $\norm{h'}_\infty \leq 1$.
	Given two $\RR$-valued random variables $X,Y$, with $\Eabs*{X} < \infty$ and $\Eabs*{Y} < \infty$ the Wasserstein distance between the laws of $X$ and $Y$, written $\dist_W(X,Y)$, is defined as
	\begin{align*}
		\dist_W(X,Y) := \sup\limits_{h \in \Lip{1}} \abs*{\E*{h(X)} - \E*{h(Y)}}.
	\end{align*}
\end{definition}

\begin{definition}\label{def:total-variation}
	Given two $\NNN$-valued random variables $X,Y$, the total variation distance between the laws of $X$ and $Y$, written $\dist_{TV}(X,Y)$, is defined as
	\begin{align*}
	\dist_{TV}(X,Y) := \sup\limits_{A \subseteq \NNN}\abs{\Prob*{X \in A} - \Prob*{Y \in A}}.
	\end{align*}
\end{definition}

Note that the topologies induced by the metrics $\dist_W$ and $\dist_{TV}$ are strictly stronger, than the one induced by convergence in distribution. 
Therefore, if a sequence $(X_n)_n$ of random variables satisfies $\lim_{n \rightarrow \infty} \dist_W(X_n,Y) = 0$ resp. $\lim_{n \rightarrow \infty} \dist_{TV}(X_n,Y) = 0$ for a random variable $Y$ then it holds, that $X_n$ converges to $Y$ in distribution, i.e. $X \os{D}{\rightarrow} Y$.

We rephrase a version of the main result from \cite{LastPeccatiSchulte2016}, a so-called second order Poincar\'{e} inequality for Poisson functionals, see also \cite[Theorem 2.13]{LastPenrose2018}, it is the main device in our proof of Theorem \ref{thm:gaussian-limit}.

\begin{theorem}\label{thm:malliavin-gauss-limit}
	Let $F \in \dom{D}$ be a Poisson functional such that $\E*{F} = 0$ and $\V*{F} = 1$.
	Define
	\begin{align*}
		\gamma_1(F) & := \int\limits_{W^3} \! 
				\rbras*{\E*{(D_{x_1,x_3}^2F)^4}\E*{(D_{x_2,x_3}^2 F)^4} \E*{(D_{x_1}F)^4}\E*{(D_{x_2}F)^4}}^{\frac{1}{4}} 
				\mu^3(\id (x_1, x_2, x_3))\\
		\gamma_2(F) & := \int\limits_{W^3} \! 
				\rbras*{\E*{(D_{x_1,x_3}^2 F)^4}\E*{(D_{x_2,x_3}^2 F)^4}}^{\frac{1}{2}}
				\mu^3(\id (x_1, x_2, x_3))\\
		\gamma_{3,N}(F) & := \int\limits_{W}
				\Eabs*{D_xF}^3
				\mu(\id x)
	\end{align*}
	and let $Z$ be a standard Gaussian random variable,
	then
	\begin{align}\label{eq:malliavin-gauss-limit}
		\dist_W(F,Z) \leq 2 \sqrt{\gamma_1(F)} + \sqrt{\gamma_2(F)} + \gamma_{3,N}(F),
	\end{align}
	where $\dist_W$ denotes the Wasserstein-distance.
\end{theorem}

In the proof of Theorem \ref{thm:poisson-limit} we use the analogue of Theorem \ref{thm:malliavin-gauss-limit} for Poisson approximation from \cite{Grygierek2019b}:

\begin{theorem}\label{thm:malliavin-poisson-limit}
	Let $F \in \dom{D}$ be an $\NNN$-valued Poisson functional.
	Define, 
	\begin{align*}
		\gamma_{3,P}(F) & := \int\limits_{W} \rbras*{\Eabs*{D_xF(D_xF-1)}^2}^\frac{1}{2} \rbras*{\Eabs*{D_xF}^2}^\frac{1}{2} \mu(\id x),
	\end{align*}
	and let $\cP(\theta)$ be a Poisson distributed random variable with expectation and variance $\theta > 0$.
	Then
	\begin{align}\label{eq:malliavin-poisson-limit}
		\dist_{TV}\rbras*{F,\cP(\theta)} \leq \frac{1-e^{-\theta}}{\theta} \rbras*{2\sqrt{\gamma_1(F)} + \sqrt{\gamma_2(F)} + \frac{\gamma_{3,P}(F)}{\theta} + \abs*{\E*{F} - \theta} + \abs*{\V*{F} - \theta}},
	\end{align}
	where $\dist_{TV}$ denotes the total variation distance.
\end{theorem}

\subsection{Simplicial Complexes}\label{sec:prelim:simcomp}

Fix an arbitrary underlying set $S$.  An (abstract) simplicial complex is a collection $\Delta$ of finite subsets of $S$, that is closed under taking subsets, i.e. 
\begin{align*}
	\forall \frF \in \Delta: \forall \frL \subset \frF : \frL \in \Delta.
\end{align*}

The non-empty elements of $\Delta$ are called faces or simplices. Additionally every non-empty subset $\frL \subseteq \frF$ of a face $\frF \in \Delta$ is called a face of $\frF$, thus the faces of the faces of $\Delta$ are faces of $\Delta$ themselves.

The number of $i$-dimensional faces of $\Delta$ will be denoted by $f_i(\Delta)$ and we note that the vector $(f_{-1}(\Delta), f_0(\Delta), f_1(\Delta), \ldots)$ is the $\bff$-vector of $\Delta$, where $f_{-1}(\Delta)$ is the Euler characteristics of $\Delta$ and $f_0(\Delta)$ denotes the number of vertices of $\Delta$, see \cite[Definition 8.16, p.245]{Ziegler1995} for more details.

The Vietoris-Rips complex is an example of a simplicial complex that arises naturally from metric spaces.

\begin{definition}[Vietoris-Rips complex]\label{def:vietoris-rips-complex}
	Let $X = (X,\dist)$ be a metric space (usually a locally finite subset of $\RR^d$) and $\delta \in (0,\infty)$.
	The Vietoris-Rips complex of $(X,\dist)$ with respect to $\delta$ (and the underlying metric $\dist$) is the abstract simplicial complex on vertex set $X$ whose $k$-simplices are all subsets $\cbras*{x_0,\ldots,x_k} \subseteq X$ with $\dist(x_i,x_j) \leq \delta$ for all $i,j \in \cbras*{0,\ldots,k}$.
\end{definition}

For more details on simplicial complexes, we refer the reader to the books \cite{Stanley1996,Munkres1984,Ziegler1995}.

\section{Moment-Decomposition for \texorpdfstring{$U$-}{U-}statistics}\label{sec:decomp}

Let $\eta$ be a Poisson point process on $\XX$ with intensity measure $\mu$. For $n \geq 1$ we consider the $U$-statistics $F$ of order $n$ with symmetric and measurable kernel $h:\XX^n \rightarrow \RR$ given by
\begin{align}\label{eq:U-staticstic}
	F := \frac{1}{n!} \sum\limits_{(y_0,\ldots,y_{n-1}) \in \eta^n_{\neq}} h(y_0,\ldots,y_{n-1}).
\end{align}

To prove our main results, we will introduce the following helpful decomposition of the $p$-th power $F^p$ for $p \in \cbras*{2,3,4}$, that allows us to apply Mecke's formula \eqref{eq:mecke} to each term in the decomposition and derive the corresponding moments of $F$.

The main idea is to split the summation over $(\eta^n_{\neq})^p$ into multiple sums over sets that are diagonal free and therefore satisfy the pairwise distinct condition needed for the index set in Mecke's formula \eqref{eq:mecke}. Identifying variables in the tuples that are assumed to be equal and accounting for all possible combinations and permutations we reduce the different cases using the symmetry $h$ to the terms given in the lemmas below.

Note that the constants $\bC(n,\cdot) \in (0,\infty)$ are combinatorially constants that do only depend on $n$ and the given indices of the corresponding sum.

\begin{notation}\label{not:ivar}
	To shorten our notation we will use $\ivarr ymn$ instead of $y_m,\ldots,y_{m+n-1}$ and $\ivar yn$ instead of $y_0,\ldots,y_{n-1}$. Further $\ivarr yn0$ resp. $\ivar y0$ indicates that no $y$-variables are used.
\end{notation}

\begin{lemma}\label{lem:decomp:2}
	For all $n \geq 1$ there existing explicit constants $\bC(n,r) \in (0,\infty)$ only depending on $n$ and $r$, such that the second moment of $F$ is given by
	\begin{align*}
		\E*{F^2} = \sum\limits_{r=0}^n \bC(n,r) 
			\smashoperator{\int\limits_{\XX^{2n-r}}} h(\ivar yn) h(\ivar yr , \ivar z{n-r}) \id \mu^{n}(\ivar yn) \id \mu^{n-r}(\ivar z{n-r}),
	\end{align*}
	where $r$ is the number of variables that are shared in both kernel functions in the integral.
\end{lemma}

Note that $\bC(n,0) = \frac{1}{n!^2}$ and the corresponding integral for $r = 0$ equals $(n!\E*{F})^2$, which directly yields the following representation for the variance:

\begin{corollary}\label{cor:decomp:var}
	For all $n \geq 1$ there existing explicit constants $\bC(n,r) \in (0,\infty)$ only depending on $n$ and $r$, such that the variance of $F$ is given by
	\begin{align*}
		\V*{F} = \sum\limits_{r=1}^{n} \bC(n,r) 
			\smashoperator{\int\limits_{\XX^{2n-r}}} h(\ivar yn) h(\ivar yr, \ivar z{n-r}) \id \mu^{n}(\ivar yn) \mu^{n-r}(\ivar z{n-r}).
	\end{align*}
\end{corollary}

For the third and fourth moment we derive similar representations involving the product of three resp. four kernel functions.

\begin{lemma}\label{lem:decomp:3}
	For all $n \geq 1$ there existing explicit constants $\bC(n,\cdot) \in (0,\infty)$ only depending on the given values, such that the third moment of $F$ is given by
	\begin{align*}
		& \E*{F^3} = \sum\limits_{r=0}^n \sum\limits_{s_Y = 0}^n \smashoperator[r]{\sum\limits_{s_Z = 0}^{\min\rbras*{\substack{n-r\phantom{{}_Y} \\ n-s_Y}}}} \bC(n,r,s_Y,s_Z)\\
			& \smashoperator[r]{\int\limits_{\XX^{3n-r-s}}} h(\ivar yn) h(\ivar yr, \ivar z{n-r}) h(\ivar y{s_Y},\ivar z{s_Z},\ivar w{n-s}) \id \mu^n(\ivar yn) \id \mu^{n-r}(\ivar z{n-r}) \id \mu^{n-s}(\ivar w{n-s}),
	\end{align*}
	where $s := s_Y + s_Z$ and the indices of the sums are denoting the number of variables that are shared in multiple kernel functions in the integral. To shorten our notation, we denote by $\min\rbras*{\substack{a \\ b}}$ the minimum of $a$ and $b$.
\end{lemma}

\begin{lemma}\label{lem:decomp:4}
	For all $n \geq 1$ there existing explicit constants $\bC(n,\cdot) \in (0,\infty)$ only depending on the given values, such that the fourth moment of $F$ is given by
	\begin{align*}
		& \E*{F^4} = \sum\limits_{r=0}^n 
			\sum\limits_{s_Y = 0}^n \sum\limits_{s_Z = 0}^{\min\rbras*{\substack{n-r\phantom{{}_Y} \\ n-s_Y}}} 
			\sum\limits_{m_Y = 0}^n \sum\limits_{m_Z = 0}^{\min\rbras*{\substack{n-r\phantom{{.}_Y} \\ n-m_Y}}} \smashoperator[r]{\sum\limits_{m_W = 0}^{\min\rbras*{\substack{n-s\phantom{{.}_{Y}-m_Z} \\ n-m_Y-m_Z}}}}  \bC(n,r,s_Y,s_Z,m_Y,m_Z,m_W)\\
			& \smashoperator[r]{\int\limits_{\XX^{4n-r-s-m}}} h(\ivar yn)h(\ivar yr,\ivar z{n-r})h(\ivar y{s_Y},\ivar z{s_Z},\ivar w{n-s})h(\ivar y{m_Y},\ivar z{m_Z},\ivar w{m_W},\ivar u{n-m})\\
			& \qquad \qquad \id \mu^n(\ivar yn) \id \mu^{n-r}(\ivar z{n-r}) \id \mu^{n-s}(\ivar w{n-s}) \id \mu^{n-m}(\ivar u{n-m}),
	\end{align*}
	where $s := s_Y + s_Z$, $m := m_Y + m_Z + m_W$ and the indices of the sums are denoting the number of variables that are shared in multiple kernel functions in the integral. To shorten our notation, we denote by $\min\rbras*{\substack{a \\ b}}$ the minimum of $a$ and $b$.
\end{lemma}

\begin{proof}[Proof: Moment-Decomposition for $U$-statistics]
	We denote the sets of $n$ variables enumerated from $0$ to $n-1$ with capital letters
	\begin{align*}
		\begin{split}
			Y(n) & = \cbras*{y_0,\ldots,y_{n-1}},\\
			Z(n) & = \cbras*{z_0,\ldots,z_{n-1}},
		\end{split}
		\begin{split}
			W(n) & = \cbras*{w_0,\ldots,w_{n-1}},\\
			U(n) & = \cbras*{u_0,\ldots,u_{n-1}},
		\end{split}
	\end{align*}
	and for $X_1,X_2 \subseteq Y(n) \cup Z(n) \cup W(n) \cup U(n)$ we denote by $\Psi(X_1,X_2)$ the set of all injective maps $\psi:X_1 \rightarrow X_1 \cup X_2$ such that $\psi(x) = x$ if $x \in X_1$ or $\psi(x) \in X_2$, i.e. it is not allowed to map an element of $X_1$ onto another element of $X_1$ but it is allowed to injectively map any element of $X_1$ to any element of $X_2$.
	
	Note that we will choose $X_1,X_2$ such that every map $\psi \in \Psi(X_1,X_2)$ will represent one possible way to choose the variables in the index set of the summation of the corresponding $p$-th power of $F$. The case $\psi(x) = x$ represents the case that the variable $x$ is not equal to any other variable in $X_2$ and therefore gets mapped onto itself and the case $\psi(x) \in X_2$ represents the case where the variables $x$ and $\psi(x)$ are equal, representing a diagonal in the Cartesian product of the index sets of the $U$-statistic.
	Further we we will use the symmetry of the product and of $h$ to define an equivalence relation on $\Psi(X_1,X_2)$ such that all elements belonging to the same equivalence class yield the same value in the decomposition.
	
	\bigskip
	
	$\mathbf{p = 1}$: Note that the expectation of $F$ can be obtained directly using Mecke's formula \eqref{eq:mecke}, since the index set already consists of pairwise distinct tuples of points, i.e.
	\begin{align*}
		\E*{F} & = \frac{1}{n!} \dsE \sum\limits_{(y_0,\ldots,y_{n-1}) \in \eta^n_{\neq}} h(y_0,\ldots,y_{n-1})\\
			& = \frac{1}{n!} \int\limits_{\XX^n} h(y_0,\ldots,y_{n-1}) \id \mu^n(y_0,\ldots,y_{n-1}).
	\end{align*}
	
	\bigskip
	
	$\mathbf{p = 2}$: For the second moment of $F$ we rewrite the product of the sums as the sums of products renaming the variables in the second factor to obtain:
	\begin{align*}
		F^2 & = \frac{1}{n!^2} \sum\limits_{(y_0,\ldots,y_{n-1}) \in \eta^n_{\neq}}\sum\limits_{(z_0,\ldots,z_{n-1}) \in \eta^n_{\neq}} h(y_0,\ldots,y_{n-1}) h(z_0,\ldots,z_{n-1})\\
			& = \frac{1}{n!^2} \sum\limits_{(y_0,\ldots,y_{n-1},z_0,\ldots,z_{n-1}) \in \eta^n_{\neq} \times \eta^n_{\neq}} h(y_0,\ldots,y_{n-1}) h(z_0,\ldots,z_{n-1}).
	\end{align*}
	The index set $\eta^n_{\neq} \times \eta^n_{\neq}$ contains tuples of points that allow non distinct choices of points. Therefore we need to decompose the index set into all possible combinations with respect to pairwise distinct choices to apply Mecke's formula \eqref{eq:mecke}.
	This yields index sets of the form $\eta^{2n-r}_{\neq}$, where $r = 0, \ldots, n$ denotes the number of variables $z$ that are equal to another variable $y$, reducing the number of points that are chosen pairwise distinct from $\eta$.
	We have to consider the injective maps $\psi \in \Psi(Z(n),Y(n))$ and note, that we lose one dimension in our index set for every variable $z$ that gets mapped to a variable $y$, since we will reuse an already existing variable.
	We denote the number of reused variables from the set $Y(n)$ by $r = r_Y(\psi) := \abs*{\image \psi \cap Y(n)}$.
	It follows that
	\begin{align*}
		F^2 = \frac{1}{n!^2} \sum\limits_{\psi \in \Psi} \sum\limits_{(y_0,\ldots,y_{n-1}, \image \psi \cap Z(n)) \in \eta^{2n-r}_{\neq}} h(y_0,\ldots,y_{n-1})h(\psi(z_0),\ldots,\psi(z_{n-1})),
	\end{align*}
	where the index will use the variables $y_0,\ldots,y_{n-1}$ and the variables $\image \psi \cap Z(n)$, that are not replaced by variables from $Y(n)$.
	Additionally we use the symmetry of $h$ to see that only the number of reused variables $r$ determines which value the second sum will attain:
	For $r = 0,\ldots, n$ we define the equivalence relation $\psi \sim_{r} \psi'$ if $r_Y(\psi') = r_Y(\psi) = r$ and obtain $n+1$ equivalence classes $[\psi_r] \in \Psi /{\sim}$ that have to be distinguished. For simplicity we will chose the representative $\psi_r$ such that $\psi(z_i) = y_i$ for all $i = 0, \ldots, r-1$, reusing the first $r$ variables in both kernels.
	We denote the cardinality of the equivalence class $[\psi_r]$ by $\abs{\psi_r}$ and observe that $\abs{\psi_r}$ only depends on $n$ and $r$. Thus we define $\bC(n,r) := \tfrac{\abs*{\psi_r}}{n!^2}$, yielding
	\begin{align}\label{eq:decomp:F2}
		F^2 = \sum\limits_{r=0}^{n} \bC(n,r) \sum\limits_{(\ivar yn, \ivar z{n-r}) \in \eta^{2n-r}_{\neq}} h(\ivar yn)h(\ivar yr, \ivar z{n-r}).
	\end{align}
	Using Mecke's formula \eqref{eq:mecke}, the claim of Lemma \ref{lem:decomp:2} is obtained directly.
	
	\bigskip	
	
	$\mathbf{p=3}$:	
	
	Similarly the proof of Lemma \ref{lem:decomp:3} is obtained, starting with the representation of $F^2$ given by \eqref{eq:decomp:F2} and multiplying by $F$. We use the the injective maps $\psi \in \Psi(W(n),Y(n) \cup Z(n-r))$. We denote the number of reused variables from $Y(n)$ resp. $Z(n-r)$ by $s_Y$ resp. $s_Z$ and define the equivalence relation $\psi \sim_{(s_Y,s_Z)} \psi'$ if and only if $s_Y(\psi') = s_Y(\psi) = s_Y$ and $s_Z(\psi') = s_Z(\psi) = s_Z$ to obtain the decomposition of $F^3$ into the summation over index sets of the form $\eta^{3n-r-s_Y-s_Z}_{\neq}$.
	
	\bigskip
	
	$\mathbf{p=4}$:	
	
	Lemma \ref{lem:decomp:4} follows by another iteration of this technique, using the injective maps $\psi \in \Psi(U(n), Y(n) \cup Z(n-r) \cup W(n-s))$ and defining the equivalence relation $\sim_{(m_Y,m_Z,m_W)}$ where $m_Y,m_Z,m_W$ denoting the number of reused variables from $Y(n)$, $Z(n-r)$ resp. $W(n-s)$.
\end{proof}

\section{Proofs of the main results}\label{sec:proof}

For all $n \geq 2$, $d \geq 1$, $\delta > 0$ and all $y_1, \ldots, y_n \in \RR^d$ we define the indicator
\begin{align*}
	\leqdelta*{y_1,\ldots,y_n} := \prod\limits_{i=1}^n \prod\limits_{j=i+1}^n \1\cbras*{\unorm*{y_j - y_i} \leq \delta}.
\end{align*}
where we set $\leqdelta*{y_1,\ldots,y_n} := 1$ for $n \leq 1$.
Additionally we combine these indicator functions with Notation \ref{not:ivar} to shorten our notation throughout this section.

\begin{remark}
	We note, that this indicator function can also be represented using the alternative condition
	\begin{align*}
		\leqdelta*{y_1,\ldots,y_n} := \1\cbras*{\max\limits_{i,j=1}^n \unorm*{y_j-y_i} \leq \delta},
	\end{align*}
	and satisfies the factorization inequality
	\begin{align}\label{eq:leqdelta:factorization}
		\leqdelta*{y_1,\ldots,y_r,y_{r+1}, \ldots y_n} \leq \leqdelta*{y_1,\ldots,y_r} \leqdelta*{y_1, \ldots, y_n}
	\end{align}
	and the argument-removal inequality
	\begin{align}\label{eq:leqdelta:removal}
		\leqdelta*{y_1,\ldots,y_r,y_{r+1}, \ldots y_n} \leq \leqdelta*{y_1,\ldots,y_r}
	\end{align}
	for all $n \geq 2$ and $r \in \cbras*{1,\ldots,n}$.
\end{remark}

Having this notations in place, the $k$-simplex counting functional $F_k$ is a $(k+1)$-order $U$-statistic with measurable and symmetric kernel $\leqdeltaop:W^{k+1} \rightarrow \cbras*{0,1}$ given by
\begin{align*}
	F_k := \frac{1}{(k+1)!} \sum\limits_{(y_0,\ldots,y_k) \in \eta^{k+1}_{\neq}} \leqdelta*{y_0,\ldots,y_k}.
\end{align*}

In the calculation of expectation and variance we will handle boundary effects, using the inner parallel set $W_{-\delta}$ of $W$ that is defined by
\begin{align*}
	W_{-\delta} := \cbras*{x \in W: \uBB{x}{\delta} \subseteq W} = \sbras*{-\tfrac{1}{2}+\delta, +\tfrac{1}{2}-\delta}^d.
\end{align*}
It is important to notice, that
\begin{align*}
	\Lambda_d(W_{-\delta}) = (1-2\delta)^d,
\end{align*}
depends on the dimension $d$ and on $\delta$. Especially, the limit for $d \rightarrow \infty$ is determined by the convergence speed of $\delta$ and has a major influence on our bounds for the expectation and variance.

\begin{remark}\label{rem:delta:convergence}
	We will choose the sequence $(\delta_d)_{d} \in (0,\infty)$ such that $\delta_d \rightarrow 0$ and
	\begin{align*}
		\lim\limits_{d \rightarrow \infty} \Lambda_d(W_{-\delta_d}) = \lim\limits_{d \rightarrow \infty} \rbras*{1-2\delta_d}^d \os{!}{=} 1 = \Lambda_d(W).
	\end{align*}
	Therefore $\delta_d$ has to decrease faster than $\frac{1}{d}$, i.e. we require
	\begin{align*}
		\lim\limits_{d \rightarrow \infty} d \delta_d = 0,
	\end{align*}
	to ensure, that the observation window related factor in the lower and upper bound has the same limit. 
	This condition can be weakened in the Gaussian case to $\lim_{d \rightarrow \infty} d \delta_d < \infty$ without changing the convergence rates presented in Section \ref{sec:proof:limits} below.
	However, if $\lim_{d \rightarrow \infty} d \delta_d = \infty$ rates have to be adjusted for the slower variance bound respecting $(1-2\delta_d)^d \rightarrow 0$ and it has to be ensured, that $\E*{F_k} \rightarrow \infty$ and $\widetilde{F_k} \in \dom{D}$ are still satisfied.
	In the Poisson case, the assumption can be removed completely, as long as the convergence of expectation and variance to the same positive constant is secured otherwise.
\end{remark}

We will use $g(d) \ll f(d)$ to indicate that $g(d)$ is of order at most $f(d)$, i.e.
\begin{align*}
	g(d) \ll f(d) & :\Leftrightarrow g(d) = \cO\rbras*{f(d)}\\
		&  \phantom{:}\Leftrightarrow \exists c > 0, d_0 > 0 : \forall d > d_0 : g(d) \leq c f(d),
\end{align*}
where $c$ and $d_0$ are constants not depending on $d$.

\subsection{Proof of Lemmas \ref{lem:expectation} and \ref{lem:variance}: Expectation and Variance}\label{sec:proof:exp-var}

The proof is divided into three steps, presented here as separate lemmas: First we will use Mecke's formula \eqref{eq:mecke} and integral transformations to obtain a bound for the expectation involving an integral that does only depend on $k$ and $d$. In the second step, we use the same technique combined with the variance decomposition given by Corollary \ref{cor:decomp:var} to obtain a bound for the variance. Finally we will calculate the exact values of the remaining integrals to complete the proof.

\begin{lemma}\label{lem:expectation:step:1}
	For $k \geq 1$ the expected number of $k$-simplices in the random Vietoris-Rips complex $\vietoris^{\infty}(\eta_d,\delta_d)$ is bounded by
	\begin{align*}
		\E*{F_k} & \geq \frac{\Lambda_d(W_{-\delta})}{(k+1)!} t (t\delta^d)^k \cI_{\dsE}(d,k),\\
		\E*{F_k} & \leq \frac{\Lambda_d(W)}{(k+1)!} t (t\delta^d)^k \cI_{\dsE}(d,k),
	\end{align*}
	where $\cI_{\dsE}(d,k)$ denotes the integral
	\begin{align*}
		\cI_{\dsE}(d,k) := \smashoperator{\int\limits_{\uBB 01^k}} \leqone*{y_1,\ldots,y_k} \id y_1 \cdots \id y_k.
	\end{align*}
\end{lemma}

\begin{proof}
	Using Mecke's formula \eqref{eq:mecke}, $\mu_d = t_d \Lambda_d$ and rewriting the indicator yields
	\begin{align*}
		\E*{F_k} & = \tfrac{1}{(k+1)!} \smashoperator{\int\limits_{W^{k+1}}} \1\cbras*{\max\limits_{i,j=0}^k \unorm*{y_j - y_i} \leq \delta} \id \mu^{k+1}(y_0,\ldots,y_k)\\
			& = \tfrac{t^{k+1}}{(k+1)!} \smashoperator[l]{\int\limits_{W}} \smashoperator[r]{\int\limits_{W^k}} \1\cbras*{\max\limits_{j=1}^k \unorm*{y_j - y_0} \leq \delta} \1\cbras*{\max\limits_{i,j=1}^k \unorm*{y_j - y_i} \leq \delta} \id y_1 \cdots \id y_k \id y_0.
	\end{align*}
	The linear transformation $y_j = y_j - y_0$, $y_0 \in W$ fixed, for all $j \in \cbras*{1,\ldots,k}$ has determinant $\det(J) = 1$, thus
	\begin{align*}
		\E*{F_k} = \tfrac{t^{k+1}}{(k+1)!}\smashoperator[l]{\int\limits_{W}} \smashoperator[r]{\int\limits_{(W-y_0)^k}} \1\cbras*{\max\limits_{j=1}^k \unorm{y_j} \leq \delta} \1\cbras*{\max\limits_{i,j=1}^k \unorm{y_j-y_i} \leq \delta} \id y_k \cdots \id y_1 \id y_0.
	\end{align*}
	The substitution $\delta y_j = y_j$ for all $j \in \cbras*{1,\ldots,k}$ has determinant $\det(J) = \delta^{dk}$, thus
	\begin{align*}
		\E*{F_k} = \tfrac{t(t\delta^d)^k}{(k+1)!} \smashoperator[l]{\int\limits_{W}} \smashoperator[r]{\int\limits_{(\delta^{-1}(W-y_0))^k}} \1\cbras*{\max\limits_{j=1}^k \unorm{y_j} \leq 1} \1\cbras*{\max\limits_{i,j=1}^k \unorm{y_j-y_i} \leq 1} \id y_k \cdots \id y_1 \id y_0.
	\end{align*}
	Using the inner parallel set to handle the boundary effects arising from $y_0$ close to $\partial W$ we obtain the lower bound given by
	\begin{align*}
		\E*{F_k} & \geq \tfrac{1}{(k+1)!} t (t\delta^d)^k \smashoperator[l]{\int\limits_{W_{-\delta}}} \smashoperator[r]{\int\limits_{(\delta^{-1}(W-y_0) \cap \uBB 01)^k}} \1\cbras*{\max\limits_{i,j=1}^k \unorm*{y_j-y_i} \leq 1} \id y_1 \cdots \id y_k \id y_0\\
			& = \tfrac{1}{(k+1)!} t (t\delta^d)^k \smashoperator[l]{\int\limits_{W_{-\delta}}} \smashoperator[r]{\int\limits_{\uBB 01^k}} \1\cbras*{\max\limits_{i,j=1}^k \unorm*{y_j-y_i} \leq 1} \id y_1 \cdots \id y_k \id y_0\\
			& = \tfrac{\Lambda_d(W_{-\delta})}{(k+1)!} t (t\delta^d)^k \smashoperator{\int\limits_{\uBB 01^k}} \1\cbras*{\max\limits_{i,j=1}^k \unorm*{y_j - y_i} \leq 1} \id y_1 \cdots \id y_k.
	\end{align*}
	Additionally, using $\delta^{-1}(W-y_0) \cap \uBB 01 \subseteq \uBB 01$ we establish the upper bound
	\begin{align*}
		\E*{F_k} & \leq \tfrac{1}{(k+1)!} t (t\delta^d)^k \smashoperator[l]{\int\limits_{W}} \smashoperator[r]{\int\limits_{\uBB 01^k}} \1\cbras*{\max\limits_{i,j=1}^k \unorm*{y_j-y_i} \leq 1} \id y_1 \cdots \id y_k \id y_0\\
			& = \tfrac{\Lambda_d(W)}{(k+1)!} t (t\delta^d)^k \smashoperator{\int\limits_{\uBB 01^k}} \1\cbras*{\max\limits_{i,j=1}^k \unorm*{y_j - y_i} \leq 1} \id y_1 \cdots \id y_k,
	\end{align*}
	which completes the proof.
\end{proof}

\begin{lemma}\label{lem:variance:step:1}
	For $k \geq 1$ there existing explicit constants $\bC(k,r)$ only depending on $k$ and $r$ such that the variance of the number of $k$-simplices in the random Vietoris-Rips complex $\vietoris^{\infty}(\eta_d,\delta_d)$ is bounded by
	\begin{align*}
		\V*{F_k} & \geq \E*{F_k} + \sum\limits_{r=1}^{k} \bC(k+1,r) \Lambda_d(W_{-\delta}) t(t\delta^d)^{2k-r+1} \cI_{\dsV}(d,k,r-1),\\
		\V*{F_k} & \leq \E*{F_k} + \sum\limits_{r=1}^{k} \bC(k+1,r) \Lambda_d(W) t(t\delta^d)^{2k-r+1} \cI_{\dsV}(d,k,r-1),
	\end{align*}
	where $\cI_{\dsV}(d,k,r)$ denotes the integral
	\begin{align*}
		\cI_{\dsV}(d,k,r) & := \smashoperator{\int\limits_{\uBB 01^{2k-r}}} \leqone*{\ivarr y1k} \leqone*{\ivarr y1r , \ivar z{k-r}} \id \ivarr y1k \id \ivar z{k-r}.
	\end{align*}
\end{lemma}

\begin{proof}
	We apply Corollary \ref{cor:decomp:var} to our $k$-simplices counting statistic $F_k$ to obtain
	\begin{align*}
		\V*{F_k} & = \sum\limits_{r=1}^{k+1} \bC(k+1,r) \smashoperator{\int\limits_{W^{2(k+1)-r}}} \leqdelta{\ivar y{k+1}} \leqdelta{\ivar yr, \ivar z{k+1-r}} \id \mu^{k+1}(\ivar y{k+1}) \id \mu^{k+1-r}(\ivar z{k+1-r}).
	\end{align*}
	For $r = k+1$ the integral is given by
	\begin{align*}
		\smashoperator{\int\limits_{W^{k+1}}} \leqdelta{\ivar y{k+1}} \leqdelta{\ivar y{k+1}} \id \mu^{k+1}(\ivar y{k+1}) = \smashoperator{\int\limits_{W^{k+1}}} \leqdelta{\ivar y{k+1}} \id \mu^{k+1}(\ivar y{k+1}),
	\end{align*}
	which is $(k+1)! \E*{F_k}$. Since $\bC(k+1,k+1) = \frac{1}{(k+1)!}$ the $r = k+1$ term in the decomposition is equal to $\E*{F_k}$.
	Therefore
	\begin{align*}
		\V*{F_k} & = \E*{F_k} + \sum\limits_{r=1}^{k} \bC(k+1,r) \smashoperator{\int\limits_{W^{2(k+1)-r}}} \1\cbras*{\max\limits_{i,j=0}^k\unorm*{y_j - y_i} \leq \delta} \1\cbras*{\max\limits_{i,j=0}^{k-r}\unorm{z_j - z_i} \leq \delta}\\
		& \qquad \qquad \times \1\cbras*{\max\limits_{i=0}^{k-r}\max\limits_{j=0}^{r-1}\unorm{y_j - z_i} \leq \delta} \id \mu^{k+1}(\ivar yk) \id \mu^{k+1-r}(\ivar z{k-r}).
	\end{align*}
	We now proceed analogously to the proof of Lemma \ref{lem:expectation:step:1} using the linear transformation $y_j = y_j - y_0$, $z_i = z_i - y_0$, $y_0 \in W$ fixed, and the substitution $\delta y_j = y_j$, $\delta z_j = z_j$ for all $j \in \cbras*{1,\ldots,k}$ and $i \in \cbras*{0,\ldots,k-r}$.
	Thus the integrals are given for $r \in \cbras*{1,\ldots,k}$ by
	\begin{align*}
		& t(t\delta^d)^{2k-r+1} \smashoperator[l]{\int\limits_{W}} \smashoperator[r]{\int\limits_{(\delta^{-1}(W-y_0))^{2k-r+1}}} \1\cbras*{\max\limits_{j=1}^k \unorm*{y_j} \leq 1 \wedge \max\limits_{i=0}^{k-r} \unorm*{z_i} \leq 1 \wedge \max\limits_{i,j=1}^k \unorm*{y_j-y_i} \leq 1}\\
		& \times \1\cbras*{\max\limits_{i,j=0}^{k-r} \unorm*{z_j - z_i} \leq 1 \wedge \max\limits_{i=0}^{k-r} \max\limits_{j=1}^{r-1} \unorm*{y_j - z_j} \leq 1}\id y_1 \cdots \id y_k \id z_0 \cdots \id z_{k-r} \id y_0\\
		=~& t(t\delta^d)^{2k-r+1} \smashoperator[l]{\int\limits_{W}} \smashoperator[r]{\int\limits_{(\delta^{-1}(W-y_0) \cap \uBB 01)^{2k-r+1}}} \1\cbras*{\max\limits_{i,j=1}^k \unorm*{y_j-y_i} \leq 1 \wedge \max\limits_{i,j=0}^{k-r} \unorm*{z_j - z_i} \leq 1}\\
		& \times \1\cbras*{\max\limits_{i=0}^{k-r} \max\limits_{j=1}^{r-1} \unorm*{y_j - z_j} \leq 1}\id y_1 \cdots \id y_k \id z_0 \cdots \id z_{k-r} \id y_0.
	\end{align*}
	Using the inner parallel set to handle the boundary effects we obtain the lower bound given by
	\begin{align*}
		\V*{F_k} \geq \E*{F_k} + & \sum\limits_{r=1}^{k} \bC(k+1,r) \Lambda_d(W_{-\delta}) t(t\delta^d)^{2k-r+1}\\
		& \quad \smashoperator[r]{\int\limits_{\uBB 01^{2k-r+1}}} \1\cbras*{\max\limits_{i,j=1}^k \unorm*{y_j-y_i} \leq 1 \wedge \max\limits_{i,j=0}^{k-r} \unorm*{z_j - z_i} \leq 1}\\
		& \quad \times \1\cbras*{\max\limits_{i=0}^{k-r} \max\limits_{j=1}^{r-1} \unorm*{y_j - z_j} \leq 1}\id y_1 \cdots \id y_k \id z_0 \cdots \id z_{k-r}.
	\end{align*}
	Additionally, we establish the upper bound
	\begin{align*}
		\V*{F_k} \leq \E*{F_k} + & \sum\limits_{r=1}^{k} \bC(k+1,r) \Lambda_d(W) t(t\delta^d)^{2k-r+1}\\
		& \quad \smashoperator[r]{\int\limits_{\uBB 01^{2k-r+1}}} \1\cbras*{\max\limits_{i,j=1}^k \unorm*{y_j-y_i} \leq 1 \wedge \max\limits_{i,j=0}^{k-r} \unorm*{z_j - z_i} \leq 1}\\
		& \quad \times \1\cbras*{\max\limits_{i=0}^{k-r} \max\limits_{j=1}^{r-1} \unorm*{y_j - z_j} \leq 1}\id y_1 \cdots \id y_k \id z_0 \cdots \id z_{k-r},
	\end{align*}
	which completes the proof.
\end{proof}

We are now left with the task of determining the exact values of the two integrals $\cI_{\dsE}$ in Lemma \ref{lem:expectation:step:1} and $\cI_{\dsV}$ in Lemma \ref{lem:variance:step:1}:

\begin{lemma}\label{lem:expectation:step:2}
	For $d \geq 1$ and $k \geq 1$:
	\begin{align*}
		\cI_{\dsE}(d,k) = (k+1)^d.
	\end{align*}
\end{lemma}

\begin{proof}
	Let us first observe that
	\begin{align*}
		\1\cbras*{\unorm*{y_i - y_j} \leq 1} = \prod\limits_{n=1}^d \1\cbras*{\abs*{y_{i,n} - y_{j,n}} \leq 1},
	\end{align*}
	where $y_{i,n}$ denotes the $n$-th component of the point $y_i \in \RR^d$. Thus
	\begin{align*}
		\cI_{\dsE}(d,k) = \cI_{\dsE}(1,k)^d = \rbras*{~\int\limits_{[-1,1]^k} \1\cbras*{\max\limits_{i,j=1}^k \abs*{y_j - y_i} \leq 1} \id y_1 \cdots \id y_k }^d,
	\end{align*}
	and we are left to show that $\cI_{\dsE}(1,k) = k+1$.
	We note that 
	\begin{align*}
		\1\cbras*{\max\limits_{i,j=1}^k \abs*{y_j - y_i} \leq 1} = \1\cbras*{\max\cbras*{y_1,\ldots,y_k} - \min\cbras*{y_1,\ldots,y_k} \leq 1},
	\end{align*}
	for all $y_1, \ldots, y_k \in [-1,1]$ and thus
	\begin{align*}
		\cI_{\dsE}(1,k) = \smashoperator{\int\limits_{[-1,1]^k}} \1\cbras*{\max\cbras*{y_1,\ldots,y_k} - \min\cbras*{y_1,\ldots,y_k} \leq 1} \id y_1 \cdots \id y_k.
	\end{align*}
	Since the integrand does only depend on the maximum and the minimum of the variables $y_1,\ldots,y_k$ we split $[-1,1]^k$ into $k(k-1)$ different regions that correspond to the different choices of the maximum and the minimum and observe that all regions yield the same contribution to the complete integral. Therefore we can assume without loss of generality that $y_1$ is the maximum, $y_2$ is the minimum and $y_3, \ldots, y_k \in [y_2, y_1]$. Therefore we calculate the integral
	\begin{align*}
		& \int\limits_{-1}^1 \int\limits_{-1}^{y_1} \1\cbras*{y_1 - y_2 \leq 1} (y_1-y_2)^{k-2} \id y_2 \id y_1
		= \int\limits_{-1}^1 \int\limits_{\max\cbras*{-1,y_1-1}}^{y_1} (y_1 - y_2)^{k-2} \id y_2 \id y_1\\
		= & \int\limits_{-1}^0 \int\limits_{-1}^{y-1} (y_1-y_2)^{k-2} \id y_2 \id y_1 + \int\limits_{0}^{1} \int\limits_{y_1-1}^{y_1} (y_1-y_2)^{k-2} \id y_2 \id y_1\\
		= & \int\limits_{-1}^0 \frac{(y_1+1)^{k-1}}{k-1} \id y_1 + \int\limits_{0}^{1} \frac{1}{k-1} \id y_1 
		= \frac{1}{k(k-1)} + \frac{1}{k-1} = \frac{k+1}{k(k-1)}.
	\end{align*}
	Multiplying by the number of regions yields $\cI_{\dsE}(1,k) = k+1$ and completes the proof.
\end{proof}

%TODO pruefe die schreibweise der Integrale im Beweis, passen die 3 bedingungen wirklich?

\begin{lemma}\label{lem:variance:step:2}
	For $d \geq 1$, $k \geq 1$ and $r \in \cbras*{0,\ldots,k}$:
	\begin{align*}
		\cI_{\dsV}(d,k,r) = \rbras*{\frac{2(k+2)(k-r)}{r+2} + r+1}^d.
	\end{align*}
\end{lemma}
	
\begin{proof}
	As in the proof of Lemma \ref{lem:expectation:step:2} we reduce the proof to the case $d=1$ and rewrite the indicators to obtain the integral
	\begin{align*}
		\cI_{\dsV}(1,k,r) & = \smashoperator[l]{\int\limits_{[-1,1]^k}} \smashoperator[r]{\int\limits_{[-1,1]^{k-r}}} \1\cbras*{\max\cbras*{\ivarr y1k} - \min\cbras*{\ivarr y1k} \leq 1}\\
		& \qquad \times \1\cbras*{\max\cbras*{\ivarr y1r, \ivar z{k-r}} - \min\cbras*{\ivarr y1r, \ivar z{k-r}} \leq 1} \id \ivar z{k-r} \id \ivarr y1k.
	\end{align*}
	For $r = 0$ the integral factorizes into $\cI_{\dsE}(1,k)^2$ and since $h^2 = h$ the case $r = k$ can be directly reduced to $\cI_{\dsE}(1,k)$ yielding the claim.
	For $r \in \cbras*{1,\ldots,k-1}$ we rearrange the order of integration to obtain
	\begin{align*}
		\smashoperator{\int\limits_{[-1,1]^r}} \1\cbras*{\max\cbras*{\ivarr y1r} - \min\cbras*{\ivarr y1r} \leq 1} \cJ\rbras*{y_1,\ldots,y_r}^2 \id \ivarr y1r,
	\end{align*}
	where we define
	\begin{align*}
		\cJ\rbras*{y_1,\ldots,y_r} :=  \smashoperator{\int\limits_{[-1,1]^{k-r}}} \1\cbras*{\max\cbras*{\ivarr y1r, \ivar z{k-r}} - \min\cbras*{\ivarr y1r, \ivar z{k-r}} \leq 1} \id \ivar z{k-r}.
	\end{align*}
	Let us first examine the condition in the indicator of $\cJ\rbras*{y_1,\ldots,y_r}$:
	We note that $\max\cbras*{\ivarr y1r, \ivar z{k-r}} - \min\cbras*{\ivarr y1r, \ivar z{k-r}} \leq 1$ is satisfied if and only if the following three conditions are satisfied at the same time:
	\begin{align*}
		\abs*{y_j - y_i} \leq 1 & \quad \forall i,j \in \cbras*{1,\ldots,r},\\
		\abs*{z_j - z_i} \leq 1 & \quad \forall i,j \in \cbras*{0,\ldots,k-r-1},\\
		\abs*{y_j - z_i} \leq 1 & \quad \forall i \in \cbras*{0,\ldots,k-r-1}, j \in \cbras*{1,\ldots,r}.
	\end{align*}
	The first condition is always satisfied since we integrate over $y_1,\ldots,y_r$ with respect to the corresponding indicator in the outer integral and the third condition is equivalent to 
	\begin{align*}
		z_i \in \sbras*{-1+\max\cbras*{0,y_1,\ldots,y_r}, 1 + \min\cbras*{0, y_1, \ldots, y_r}}, \quad \forall i \in \cbras*{0,\ldots,k-r-1}.
	\end{align*}
	We define the positive part of the maximum and the negative part of the minimum by
	\begin{align*}
		y_{\max} & := \max\cbras*{0,y_1,\ldots,y_r}\\
		y_{\min} & := \min\cbras*{0,y_1,\ldots,y_r}.
	\end{align*}
	Assuming $y_{\max} - y_{\min} \leq 1$ it follows that
	\begin{align*}
		\cJ\rbras*{y_1,\ldots,y_r} = \smashoperator{\int\limits_{[-1+y_{\max}, 1+y_{\min}]^{k-r}}} \1\cbras*{\max\cbras*{\ivar z{k-r}} - \min\cbras*{\ivar z{k-r}} \leq 1} \id \ivar z{k-r}.
	\end{align*}
	Let us first consider the case $r = k-1$ implying $k-r = 1$:
	It follows immediately that $\cJ(y_1,\ldots,y_r) = 2 + y_{\min} - y_{\max}$. 
	For $r \in \cbras*{1,\ldots,k-2}$ we continue similar to the proof of Lemma \ref{lem:expectation:step:2}:
	We split the domain of the integral into $(k-r)(k-r-1)$ regions that correspond to the different choices of the maximum and the minimum. Without loss of generality we assume that $z_0$ is the maximum and $z_1$ is the minimum, which implies that $z_2, \ldots, z_{k-r} \in [z_1,z_0]$.
	Therefore we obtain
	\begin{align*}
		\cJ(y_1,\ldots,y_r) & = \int\limits_{-1+y_{\max}}^{1+y_{\min}} \int\limits_{-1+y_{\max}}^{z_0} \1\cbras*{z_0 - z_1 \leq 1} (z_0 - z_1)^{k-r-2} \id z_1 \id z_0\\
			& = 1 + (k-r)\rbras*{1-\rbras*{y_{\max} - y_{\min}}},
	\end{align*}
	which also represents the equation for $r = k-1$.
	Our next objective is to evaluate the integral
	\begin{align*}
		\cI_{\dsV}(1,k,r) = & \smashoperator{\int\limits_{[-1,1]^r}} \1\cbras*{\max\cbras*{\ivarr y1r} - \min\cbras*{\ivarr y1r} \leq 1}\\
		& \quad \rbras*{1 + (k-r)\rbras*{1-\rbras*{\max\cbras*{0,\ivarr y1r} - \min\cbras*{0,\ivarr y1r}}}}^2 \id \ivarr y1r.
	\end{align*}
	For $r = 1$ we simply derive
	\begin{align*}
		\cI_{\dsV}(1,k,1) & = \smashoperator{\int\limits_{-1}^1} \rbras*{1 + (k-1)\rbras*{1-\rbras*{\max\cbras*{0,y} - \min\cbras*{0,y}}}}^2 \id y\\
		& = \smashoperator{\int\limits_{-1}^0} \rbras*{1+(k-1)(1+y)}^2 \id y + \smashoperator{\int\limits_{0}^1} \rbras*{1+(k-1)(1-y)}^2 \id y = \frac{2}{3} \rbras*{k^2 + k + 1},
	\end{align*}
	which is the desired result.
	For $r \in \cbras*{2,\ldots,k-1}$ we split the domain of the integral into $r(r-1)$ regions. Without loss of generality we assume that $y_1$ is the maximum and $y_2$ the minimum, which implies that $y_3,\ldots,y_r \in [y_2,y_1]$.
	Thus 
	\begin{align*}
		& \smashoperator[l]{\int\limits_{-1}^1} \smashoperator[r]{\int\limits_{-1}^{y_1}} \1\cbras*{y_1 - y_2 \leq 1} (y_1 - y_2)^{r-2} J(y_1,\ldots,y_r)^2 \id y_2 \id y_1\\
		= & \smashoperator[l]{\int\limits_{-1}^1} \smashoperator[r]{\int\limits_{\max\cbras*{-1,y_1-1}}^{y_1}} (y_1 - y_2)^{r-2} \rbras*{1 + (k-r) \rbras*{1-\rbras*{\max\cbras*{0,y_1} - \min\cbras*{0,y_2}}}}^2 \id y_2 \id y_1\\
		= & \smashoperator[l]{\int\limits_{-1}^0} \smashoperator[r]{\int\limits_{-1}^{y_1}} (y_1 - y_2)^{r-2} \rbras*{1 + (k-r) \rbras*{1-\rbras*{0 - y_2}}}^2 \id y_2 \id y_1\\
		& + \smashoperator[l]{\int\limits_{0}^1} \smashoperator[r]{\int\limits_{y_1-1}^{0}} (y_1 - y_2)^{r-2} \rbras*{1 + (k-r) \rbras*{1-\rbras*{y_1 - y_2}}}^2 \id y_2 \id y_1\\
		& + \smashoperator[l]{\int\limits_{0}^1} \smashoperator[r]{\int\limits_{0}^{y_1}} (y_1 - y_2)^{r-2} \rbras*{1 + (k-r) \rbras*{1-\rbras*{y_1 - 0}}}^2 \id y_2 \id y_1\\
		= & \frac{(2(k-r)+3)r+2(k-r+1)^2+r^2}{(r-1)r(r+2)}
	\end{align*}
	Multiplying by the number of regions yields the claim for $\cI_{\dsV}(1,k,r)$, which completes the proof, since $\cI_{\dsV}(d,k,r) = \cI_{\dsV}(1,k,r)^d$.
\end{proof}

\begin{proof}[Proof of Lemmas \ref{lem:expectation} and \ref{lem:variance}]
	Combining Lemma \ref{lem:expectation:step:1} with Lemma \ref{lem:expectation:step:2} yields the bound for the expectation, Lemma \ref{lem:expectation}. 
	Combining Lemma \ref{lem:variance:step:1} with Lemma \ref{lem:variance:step:2} yields the bound for the variance, Lemma \ref{lem:variance}.
\end{proof}

\subsection{Bounds for first and second order Malliavin-Derivatives}
\label{sec:proof:malliavin-derivative}

The first order difference operator of our $k$-simplex counting functional is a $U$-statistics of order $k$, given for all $x \in W$ by
\begin{align*}
	D_xF_k = \frac{1}{k!} \sum\limits_{(y_0,\ldots,y_{k-1}) \in \eta^k_{\neq}} \leqdelta{y_0,\ldots,y_{k-1},x}.
\end{align*}
The second order difference operator is a $U$-statistics of order $k-1$, given for all $x_1,x_2 \in W$ by
\begin{align*}
	D_{x_1,x_2}F_k = \frac{1}{(k-1)!} \sum\limits_{(y_0,\ldots,y_{k-2}) \in \eta^{k-1}_{\neq}} \leqdelta{y_0,\ldots,y_{k-2},x_1,x_2},
\end{align*}
if $k \geq 2$ and $D_{x_1,x_2}F_k = \leqdelta{x_1,x_2}$ if $k = 1$, see for instance \cite[Lemma 3.3]{ReitznerSchulte2013}.

The crucial part in the application of the Malliavin-Stein method, Theorems \ref{thm:malliavin-gauss-limit} and \ref{thm:malliavin-poisson-limit}, is the control over the moments of the difference operators that are used in $\gamma_1$, $\gamma_2$ and $\gamma_{3,N}$ resp. $\gamma_{3,P}$. In this section, we will prove the following bounds:

\begin{theorem}\label{thm:derivative:bound}
	Let $k \geq 1$ and $d \geq 1$:
	\begin{enumerate}
		\item For all $p \in \cbras*{2,3,4}$ there existing constants $\bD_p(k)$ only depending on $k$ and $p$ such that for all $x \in W$ it holds, that
			\begin{align}\label{eq:derivative:bound:DxF}
				\E*{(D_xF_k)^p} & \leq \bD_p(k) \sum\limits_{q=k}^{pk} (t\delta^d)^q\rbras*{(k+1)\rbras*{\tfrac{q-k}{p-1} + 1}^{p-1}}^d.
			\end{align}
		\item There exists a constant $\bD_4^*(k)$ only depending on $k$ such that for all $x \in W$ it holds, that
			\begin{align}\label{eq:derivative:bound:DxFDxF-1}
				\E*{\rbras[\big]{(D_xF_k)(D_xF_k-1)}^2} \leq \bD_4^*(k) \sum\limits_{q=k+1}^{4k} (t\delta^d)^q\rbras*{(k+1)\rbras*{\tfrac{q-k}{3}+1}^3}^d.
			\end{align}
		\item There exists a constant $\bD_4'(k)$ only depending on $k$ such that for all $x_1,x_2 \in W$ it holds, that
			\begin{align}\label{eq:derivative:bound:DxxF}
				\E*{(D_{x_1,x_2}F_k)^4} \leq \leqdelta*{x_1,x_2} \bD_4'(k) \sum\limits_{q=k-1}^{4(k-1)} (t\delta^d)^q \rbras[\Big]{k \rbras*{\tfrac{q-(k-1)}{3}+1}^3}^{d}.
			\end{align}
	\end{enumerate}
\end{theorem}

We have divided the proof into a sequence of lemmas. 
At first we will use the moment decomposition for $U$-statistics to obtain bounds, that only involve integrals that depend on $d$, $k$ and the indices given by the decomposition, see Section \ref{sec:decomp}. Therefore, for all $d \geq 1$, $k \geq 1$ we define the integrals
\begin{align}
	\label{eq:cI_2}
	\cI_2 & := \smashoperator[r]{\int\limits_{{\uBB 01}^{2k-r}}} \leqone*{\ivar yk} \leqone*{\ivar yr, \ivar z{k-r}} \id \ivar yk \id \ivar z{k-r}, \\
	\label{eq:cI_3}
	\cI_3 & := \smashoperator[r]{\int\limits_{{\uBB 01}^{3k-r-s}}} \leqone*{\ivar yk} \leqone*{\ivar yr, \ivar z{k-r}} \leqone*{\ivar y{s_Y}, \ivar z{s_Z}, \ivar w{k-s}} \id \ivar yk \id \ivar z{k-r} \id \ivar w{k-s},\\
	\label{eq:cI_4}
	\begin{split}
	\cI_4 & := \smashoperator[r]{\int\limits_{{\uBB 01}^{4k-r-s-m}}} \leqone*{\ivar yk} \leqone*{\ivar yr, \ivar z{k-r}} \leqone*{\ivar y{s_Y}, \ivar z{s_Z}, \ivar w{k-s}}\\
	& \qquad \qquad \times \leqone*{\ivar y{m_Y}, \ivar z{m_Z}, \ivar w{m_W}, \ivar u{k-m}} \id \ivar yk \id \ivar z{k-r} \id \ivar w{k-s} \id \ivar u{k-m},
	\end{split}
\end{align}
where $s := s_Y + s_Z$, $m := m_Y + m_Z + m_W$ and the indices $r,s_Y,s_Z,m_Y,m_Z,m_W$ are given according to the summations in the corresponding moment-decomposition.

\begin{lemma}\label{lem:DF:step:1}
	For all $k \geq 1$, $p \in \cbras*{2,3,4}$ there existing constants $\bD_p(k)$ only depending on $k$ and $p$ such that for all $x \in W$ it holds, that
	\begin{align*}
		\E*{(D_xF_k)^2} & \leq \bD_2(k) \sum (t\delta^d)^{2k-r} \cI_2(d,k,r),\\
		\E*{(D_xF_k)^3} & \leq \bD_3(k) \sum \ldots \sum (t\delta^d)^{3k-r-s} \cI_3(d,k,r,s_Y,s_Z),\\
		\E*{(D_xF_k)^4} & \leq \bD_4(k) \sum \ldots \sum (t\delta^d)^{4k-r-s-m} \cI_4(d,k,r,s_Y,s_Z,m_Y,m_Z,m_W),
	\end{align*}
	where the summations runs over the indices $r,s_Y,s_Z,m_Y,m_Z,m_W$ given in the corresponding moment-decomposition of the $k$-order $U$-statistics $D_xF_k$.
\end{lemma}

\begin{proof}
	Fix $k \geq 1$ and $x \in W$. 
	We use the moment-decomposition for $U$-statistics, Lemma \ref{lem:decomp:2}, on $D_xF$ to obtain
	\begin{align*}
		\E*{(D_xF_k)^2} = \sum\limits_{r=0}^k \bC(k,r) \smashoperator{\int\limits_{W^{2k-r}}} \leqdelta{\ivar yk,x}\leqdelta{\ivar yr, \ivar z{k-r},x} \id \mu^k(\ivar yk) \id \mu^{k-r}(\ivar z{k-r}).
	\end{align*}
	The linear transformation $y_j = y_j - x$ and $z_i = z_i - x$ and the substitution $\delta y_j = y_j$, $\delta z_i = z_i$ for all $j \in \cbras*{0,\ldots,k-1}$ and $i \in \cbras*{0,\ldots,k-r-1}$, similar to the proof of Lemma \ref{lem:expectation:step:1}, yields
	\begin{align*}
		\E*{(D_xF_k)^2} = & \sum\limits_{r=0}^k \bC(k,r) (t\delta^d)^{2k-r} \smashoperator{\int\limits_{(\delta^{-1}(W-x) )^{2k-r}}} \1\cbras*{\max\limits_{j=0}^{k-1} \unorm*{y_j} \leq 1 \wedge \max\limits_{i=0}^{k-r-1} \unorm{z_i} \leq 1}\\
		& \times \leqone*{\ivar yk} \cdot \leqone*{\ivar yr, \ivar z{k-r}} \id \ivar y{k} \id \ivar z{k-r}.
	\end{align*}
	We set $D_2(k) := \max_{r=0}^k \bC(k,r)$ and use the first indicator to bound the domain of integration, thus
	\begin{align*}
		\E*{(D_xF_k)^2} \leq \bD_2(k) \sum\limits_{r=0}^k (t\delta^d)^{2k-r} \smashoperator{\int\limits_{{\uBB 01}^{2k-r}}} \leqone*{\ivar yk} \cdot \leqone*{\ivar yr, \ivar z{k-r}} \id \ivar yk \id \ivar z{k-r},
	\end{align*}
	which establishes the formula, since the integral does not depend on $x$ anymore.
	
	In the same manner we use the moment-decomposition for $U$-statistics, Lemma \ref{lem:decomp:3} and \ref{lem:decomp:4}, on $D_xF$ to derive 
	\begin{align*}
		\E*{(D_xF_k)^3} \leq \bD_3(k) \sum \ldots \sum (t\delta^d)^{3k-r-s} \cI_3(d,k,r,s_Y,s_Z),
	\end{align*}
	and
	\begin{align*}
		\E*{(D_xF_k)^4} \leq \bD_4(k) \sum \ldots \sum (t\delta_d)^{4k-r-s-m} \cI_4(d,k,r,s_Y,s_Z,m_Y,m_Z,m_W),
	\end{align*}
	which completes the proof.
\end{proof}

\begin{lemma}\label{lem:DF:step:1-2}
	For all $k \geq 2$ there exists a constant $D_4'(k)$ only depending on $k$ such that for all $x_1,x_2 \in W$ it holds, that
	\begin{align*}
		\E*{(D_{x_1,x_2}F_k)^4} \leq & \1\cbras*{\unorm*{x_1 - x_2} \leq \delta} D_4'(k)\\
		& \times  \sum \ldots \sum (t \delta^d)^{4(k-1)-r-s-m} \cI_4(d,k-1,r,s_Y,s_Z,m_Y,m_z,m_W),
	\end{align*}
	where the summation runs over the indices $r,s_Y,s_Z,m_Y,m_Z,m_W$ given in the corresponding moment-decomposition of the $(k-1)$-order $U$-statistics $D_{x_1,x_2}F_k$.
	For $k = 1$ it holds, that $\E*{(D_{x_1,x_2}F_1)^4} = \1\cbras*{\unorm*{x_1 - x_2} \leq \delta}$.
\end{lemma}

\begin{proof}
	Since $D_{x_1,x_2}F_1 = \1\cbras*{\unorm*{x_1 - x_2} \leq \delta}$, the claim for $k=1$ follows immediately.
	Fix $k \geq 2$ and $x_1,x_2 \in W$. We use the moment-decomposition for $U$-statistics, Lemma \ref{lem:decomp:4}, on $D_{x_1,x_2}F_k$ to obtain
	\begin{align*}
		& \E*{(D_{x_1,x_2}F_k)^4} = \sum \ldots \sum \bC(k-1,r,s_Y,s_Z,m_Y,m_Z,m_W)\\
		& \times \smashoperator{\int\limits_{W^{4(k-1)-r-s-m}}} \leqdelta*{\ivar y{k-1}, x_1,x_2} \leqdelta*{\ivar yr,\ivar z{k-1-r}, x_1,x_2} \leqdelta*{\ivar y{s_Y},\ivar z{s_Z}, \ivar w{k-1-s}, x_1,x_2}\\
		& \quad \leqdelta*{\ivar y{m_Y}, \ivar z{m_Z}, \ivar w{m_W}, \ivar u{k-1-m}, x_1, x_2} \id \mu^{4(k-1)-r-s-m}(y,z,w,u).
	\end{align*}
	Using the factorization inequality \eqref{eq:leqdelta:factorization} and the argument-removal inequality \eqref{eq:leqdelta:removal} for $\leqdelta{\cdot, x_1,x_2}$, yields 
	\begin{align*}
		& \E*{(D_{x_1,x_2}F_k)^4} \leq \sum \ldots \sum \bC(k-1,r,s_Y,s_Z,m_Y,m_Z,m_W)\\
		& \times \leqdelta*{x_1,x_2} \smashoperator{\int\limits_{W^{4(k-1)-r-s-m}}} \leqdelta*{\ivar y{k-1},x_1} \leqdelta*{\ivar yr,\ivar z{k-1-r}, x_1} \leqdelta*{\ivar y{s_Y},\ivar z{s_Z}, \ivar w{k-1-s}, x_1}\\
		& \quad \leqdelta*{\ivar y{m_Y}, \ivar z{m_Z}, \ivar w{m_W}, \ivar u{k-1-m}, x_1} \id \mu^{4(k-1)-r-s-m}(y,z,w,u).
	\end{align*}
	Using the linear transformation $y = y-x_1$, $z = z-x_1$, $u = u-x_1$ and $w = w-x_1$ and the substitution $\delta y = y$, $\delta z = z$, $\delta u = u$ and $\delta w = w$ for all variables $y,z,u,w$ in the integral, analysis similar to the proof of Lemma \ref{lem:DF:step:1} yields the desired result.
\end{proof}

In the next step we derive a bound for the integrals depending on the indices in the moment-decomposition:

\begin{lemma}\label{lem:DF:step:2}
	For all $d \geq 1$, $k \geq 1$ and all choices of indices $r,s_Y,s_Z,m_Y,m_Z,m_W \in \cbras*{0,\ldots,k}$ we have
	\begin{align*}
		\cI_2(d,k,r) & \leq \rbras*{(k+1)(k-r+1)}^d\\
		\cI_3(d,k,r,s_Y,s_Z) & \leq \rbras*{(k+1)(k-r+1)(k-s+1)}^d\\
		\cI_4(d,k,r,s_Y,s_Z,m_Y,m_Z,m_W) & \leq \rbras*{(k+1)(k-r+1)(k-s+1)(k-m+1)}^d,
	\end{align*}
	where $s := s_Y + s_Z$ and $m := m_Y + m_Z + m_W$.
\end{lemma}

\begin{proof}
	Similarly to the proof of Lemma \ref{lem:expectation:step:2}, we factorize the integral to obtain
	\begin{align*}
		\cI_2(d,k,r) = \cI_2(1,k,r)^d.
	\end{align*}
	Therefore we only have to consider the case $d=1$:
	\begin{align*}
		\cI_2(1,k,r) & = \smashoperator{\int\limits_{{[-1,1]}^{2k-r}}} \leqone*{\ivar yk} \leqone*{\ivar yr, \ivar z{k-r}} \id \ivar yk \id \ivar z{k-r}.
	\end{align*}
	Using \eqref{eq:leqdelta:factorization} and \eqref{eq:leqdelta:removal} we obtain
	\begin{align*}
		\leqone*{\ivar yr, \ivar z{k-r}} \leq \leqone*{\ivar yr} \cdot \leqone*{\ivar z{k-r}},
	\end{align*}
	for all $r \in \cbras*{0,\ldots,k}$.
	Hence
	\begin{align*}
		\cI_2(1,k,r) & \leq \smashoperator{\int\limits_{{[-1,1]}^{2k-r}}} \leqone*{\ivar yk} \leqone*{\ivar yr} \leqone*{\ivar z{k-r}} \id \ivar yk \id \ivar z{k-r}.
	\end{align*}
	Since $\leqone*{\ivar yk} \leqone*{\ivar yr} = \leqone*{\ivar yk}$ this integral factorizes into two integrals separating the $y$ and $z$ variables. Thus
	\begin{align*}
		\cI_2(1,k,r) \leq \int\limits_{[-1,1]^k} \leqone*{\ivar yk} \id \ivar yk \times \int\limits_{[-1,1]^{k-r}} \leqone*{\ivar z{k-r}} \id \ivar z{k-r}.
	\end{align*}
	The claim follows directly from Lemma \ref{lem:expectation:step:2}, since the two factors are given by $\cI_{\dsE}(1,k)$ and $\cI_{\dsE}(1,k-r)$.
	In the same manner we factorize the integrals $\cI_3$ and $\cI_4$.
\end{proof}

Finally, we simplify the bounds given by the previous lemma using only the number of variables in the integral, removing the dependencies on the specific choice of indices:

\begin{lemma}\label{lem:DF:step:3}
	For all $d \geq 1$, $k \geq 1$ and $p \in \cbras*{2,3,4}$ we denote by $q_p \in \cbras*{k,\ldots,pk}$ the number of variables in the integral $\cI_p$.
	\begin{enumerate}
		\item For all indices $r$ such that $2k-r = q_2$ it holds, that
			\begin{align}\label{eq:cI_2:opt}
				\cI_2(k,r) \leq \rbras*{(k+1)(q_2-k+1)}^d.
			\end{align}
		\item For all indices $r, s_Y, s_Z$ such that $3k-r-s = q_3$ it holds, that
			\begin{align}\label{eq:cI_3:opt}
				\cI_3(k,r,s_Y,s_Z) \leq \rbras*{(k+1)\rbras*{\frac{q_3-k}{2}+1}^2}^d	.
			\end{align}
		\item For all indices $r, s_Y, s_Z, m_Y, m_Z, m_W$ such that $4k-r-s-m = q_4$ it holds, that
			\begin{align}\label{eq:cI_4:opt}
				\cI_4(k,r,s_Y,s_Z,m_Y,m_Z,m_W) \leq \rbras*{(k+1)\rbras*{\frac{q_4-k}{3}+1}^3}^d.
			\end{align}
	\end{enumerate}
\end{lemma}

\begin{proof}
	We give the proof only for the case $p = 4$; the other cases are similar.
	We note that it is sufficient to show the claim for $d = 1$: For fixed $k \geq 1$ and $q_4 \in \cbras*{k,\ldots,4k}$ we define $g_3:[0,k]^3 \rightarrow \RR$ by
	\begin{align*}
		g_3(r,s,m) := (k+1)(k-r+1)(k-s+1)(k-m+1).
	\end{align*}
	Maximizing over $[0,k]^3$, with respect to the condition $F(r,s,m) := 4k-r-s-m - q_4 \os{!}{=} 0$ yields the maximal value for 
	\begin{align*}
		r = s = m = \frac{4k-q_4}{3}.
	\end{align*}
	Thus, using Lemma \ref{lem:DF:step:2}, we have for all indices $r,s_Y,s_Z,m_Y,m_Z,m_W \in \cbras*{0,\ldots,k}$, with $4k-r-s-m = q_4$ the bound
	\begin{align*}
		\cI_4(k,r,s_Y,s_Z,m_Y,m_Z,m_W) \leq g\rbras*{\frac{4k-q_4}{3},\frac{4k-q_4}{3},\frac{4k-q_4}{3}},
	\end{align*}
	which is our claim.
\end{proof}

We are now in a position, to show the main result of this section, Theorem \ref{thm:derivative:bound}:

\begin{proof}[Proof of Theorem \ref{thm:derivative:bound} a) and c)]
	We give the proof only for the case $p = 4$; the other cases are similar.
	We reorganize the summation in the bound given by Lemma \ref{lem:DF:step:1} according to the number $q$ of variables in the integral $\cI_4$, i.e. 
	\begin{align*}
		& \E*{(D_xF_k)^4} \leq \bD_4(k) \sum \ldots \sum (t\delta_d)^{4k-r-s-m} \cI_4(d,k,r,s_Y,s_Z,m_Y,m_Z,m_W)\\
		& = \bD_4(k) \sum\limits_{q=k}^{4k} \sum \ldots \sum \1\cbras*{4k-r-s-m = q} (t\delta^d)^q \cI_4(d,k,r,s_Y,s_Z,m_Y,m_Z,m_W).
	\end{align*}
	Applying \eqref{eq:cI_4:opt}, yields
	\begin{align*}
		\E*{(D_xF_k)^4} \leq \bD_4(k) \sum\limits_{q=k}^{4k} (t\delta^d)^q \rbras*{(k+1)\rbras{\tfrac{q-k}{3}+1}^3}^d \sum \ldots \sum \1\cbras*{4k-r-s-m = q}.
	\end{align*}
	Finally, the last summations over the indicator yields a constant that does only depend on $k$ and $q$.
	Thus we redefine the constant $\bD_4(k)$ and the claim a) follows. The proof of claim c) is similar.
\end{proof}

\begin{proof}[Proof of Theorem \ref{thm:derivative:bound} b)]
	Since
	\begin{align}\label{eq:thm:derivative:bound:proof:DxFDxF-1:1}
		\E*{\rbras[\big]{(D_xF)(D_xF-1)}^2} & = \E*{(D_xF)^4} - 2\E*{(D_xF)^3} + \E*{(D_xF)^2}
	\end{align}
	we use the moment decomposition for $U$-statistics, Lemmas \ref{lem:decomp:2}, \ref{lem:decomp:3} and \ref{lem:decomp:4} on $D_xF$ to obtain the decomposition for $\E{\rbras{(D_xF)(D_xF-1)}^2}$.
	Recall that the decomposition of the second moment is given by
	\begin{align*}
		\E*{(D_xF)^2} = \sum\limits_{r=0}^k \bC(k,r) \smashoperator{\int\limits_{W^{2k-r}}} \leqdelta{\ivar yk,x}\leqdelta{\ivar yr, \ivar z{k-r},x} \id \mu^k(\ivar yk) \id \mu^{k-r}(\ivar z{k-r}).
	\end{align*}
	For $r = k$ the integral is equal to the integral in $\E*{D_xF}$ an further $\bC(k,k) = \tfrac{1}{k!}$, thus we rewrite this decomposition into
	\begin{align*}
		\E*{(D_xF)^2} & = \E*{(D_xF)}\\
		& \qquad + \sum\limits_{r=0}^{k-1} \bC(k,r) \smashoperator{\int\limits_{W^{2k-r}}} \leqdelta{\ivar yk,x)}\leqdelta{\ivar yr, \ivar z{k-r},x} \id \mu^k(\ivar yk) \id \mu^{k-r}(\ivar z{k-r}).
	\end{align*}
	Similarly, for $r=k$, $s_Y=k$, which forces us to set $s_Z=0$ by definition, the term in the decomposition of the third moment is equal to $\E*{D_xF}$.
	For $r=k$, $s_Y=k$, $m_Y=k$, which forces us to set $s_Z=0$, $m_Z = 0$ and $m_W = 0$, the term in the decomposition of the fourth moment is also equal to $\E*{D_xF}$.
	Therefore, these terms cancel out in the decomposition of \eqref{eq:thm:derivative:bound:proof:DxFDxF-1:1}. We further note, that these combinations are the only choices of indices in the decomposition yielding integrals that do only involve $k$ distinct variables. It follows immediately, that all other integrals appearing in the decomposition have at least $k+1$ distinct variables. Denoting the integrals in Lemmas \ref{lem:decomp:2}, \ref{lem:decomp:3} and \ref{lem:decomp:4} by $\cJ_2$, $\cJ_3$ resp. $\cJ_4$, we use the decomposition to rewrite the mixed moments.
	Since $D_xF_k \geq 0$ for all $x \in \RR^d$ we obtain the bound
	\begin{align*}
		& \E*{\rbras[\big]{(D_xF)(D_xF-1)}^2} = \sum\limits_{q=k+1}^{2k} \sum \1\cbras*{2k-r = q} \bC(k,r) \cJ_2(d,k,r)\\
		& - \sum\limits_{q=k+1}^{3k} \sum \ldots \sum \1\cbras*{3k-r-s = q} \bC(k,r,s_Y,s_Z) \cJ_3(k,r,s_Y,s_Z)\\
		& + \sum\limits_{q=k+1}^{4k} \sum \ldots \sum \1\cbras*{4k-r-s-m = q} \bC(k,r,s_Y,s_Z,m_Y,m_Z,m_W)\\
		& \qquad \qquad \times \cJ_4(d,k,r,s_Y,s_Z,m_Y,m_Z,m_W)\\
		\leq & \sum\limits_{q=k+1}^{2k} \sum \1\cbras*{2k-r = q} \bC(k,r) \cJ_2(d,k,r)\\
		& \sum\limits_{q=k+1}^{4k} \sum \ldots \sum \1\cbras*{4k-r-s-m = q} \bC(k,r,s_Y,s_Z,m_Y,m_Z,m_W)\\
		& \qquad \qquad \times \cJ_4(d,k,r,s_Y,s_Z,m_Y,m_Z,m_w)
	\end{align*}
	where we just removed the negative term to simplify the upper bound.
	Analysis similar to that in the proof of Theorem \ref{thm:derivative:bound} a) before shows
	\begin{align*}
		& \E*{\rbras[\big]{(D_xF)(D_xF-1)}^2}\\
		\leq ~& \bD_2(k) \sum\limits_{q=k+1}^{2k} (t\delta^d)^q \rbras*{(k+1)(q-k+1)}^d + \bD_4(k) \sum\limits_{q=k+1}^{4k} (t\delta^d)^q \rbras*{(k+1)\rbras*{\tfrac{q-k}{3}+1}^3}^d\\
		\leq ~& \rbras*{\bD_2(k) + \bD_4(k)} \sum\limits_{q=k+1}^{4k} (t\delta^d)^q \rbras*{(k+1)\rbras*{\tfrac{q-k}{3}+1}^3}^d,
	\end{align*}
	since $(q-k+1) \leq (\tfrac{q-k}{3}+1)^3$, defining $\bD_4^*(k) := \bD_2(k) + \bD_4(k)$ completes the proof.
\end{proof}

\subsection{Proof of Theorem \ref{thm:gaussian-limit} and \ref{thm:poisson-limit}: Gaussian and Poisson Limit}
\label{sec:proof:limits}

Let us first investigate the asymptotic behavior of the variance $\V*{F_k}$ in the three different phases determined by the limit of the expectation $\E*{F_k}$. Using the bound given by Lemma \ref{lem:variance} we obtain
\begin{align*}
	\E*{F_k} + (1-2\delta)^d R(d,k,t,\delta) \leq \V*{F_k} \leq \E*{F_k} + R(d,k,t,\delta),
\end{align*}
where we defined
\begin{align*}
	R(d,k,t,\delta) := t(t\delta^d)^k \sum\limits_{r = 1}^k \bC(k,r) (t\delta^d)^{k+1-r} \rbras*{\tfrac{2(k+2)(k+1-r)}{r+1}+r}^d,
\end{align*}
and note that $R(d,k,t,\delta) \geq 0$ holds for $d \geq 1$, $k \geq 1$, $t > 0$ and $\delta > 0$.

\begin{lemma}\label{lem:variance:limit}
	For all $k \geq 1$ the variance of the $k$-simplex counting functional satisfies
	\begin{align*}
		\lim\limits_{d \rightarrow \infty} \V*{F_k} = \lim\limits_{d \rightarrow \infty} \E*{F_k},
	\end{align*}
	 if the limit of the expectation is infinite \eqref{eq:phase:1}, a positive constant  \eqref{eq:phase:2} or zero \eqref{eq:phase:3}.
\end{lemma}

\begin{proof}
	Let us first assume that the limit of the expectation is either a positive constant \eqref{eq:phase:2} or zero \eqref{eq:phase:3}. Thus we have
	\begin{align*}
		\lim\limits_{d \rightarrow \infty} \tfrac{1}{(k+1)!} t(t\delta^d)^k (k+1)^d = \theta \in [0,\infty)
	\end{align*}
	and it follows, that
	\begin{align*}
		R(d,k,t,\delta) & = t(t\delta^d)^k \tfrac{(k+1)^d}{(k+1)^d} \sum\limits_{r=1}^k \bC(k,r) \rbras*{(t \delta^d)^k\tfrac{(k+1)^d}{(k+1)^d}}^{\frac{k+1-r}{k}} \rbras*{\frac{2(k+2)(k+1-r)}{r+1}+r}^d\\
		& = t(t\delta^d)^k (k+1)^d \sum\limits_{r=1}^k \bC(k,r) \rbras*{(t\delta^d)^k (k+1)^d}^{\frac{k+1-r}{k}} \rbras*{\tfrac{2(k+2)(k+1-r)+r(r+1)}{(r+1)(k+1)(k+1)^{\frac{k+1-r}{k}}}}^d.
	\end{align*}
	We define the function $g_k:[0,\infty) \rightarrow [0,\infty)$ by
	\begin{align*}
		g_k(r) := \tfrac{2(k+2)(k+1-r)+r(r+1)}{(r+1)(k+1)(k+1)^{\frac{k+1-r}{k}}}
	\end{align*}
	and a straightforward calculation shows that $g_k''(r) \geq 0$ implying convexity of $g_k$.
	Using $g_k(1) = 1$ and $g_k(k+1) = 1$ we obtain
	\begin{align*}
		g_k(\alpha + (1-\alpha)(k+1)) \leq \alpha g_k(1) + (1-\alpha)g_k(k+1) \leq 1,
	\end{align*}
	for all $\alpha \in [0,1]$ which implies $g_k(r) \in [0,1]$ for all $r \in \cbras*{1,\ldots,k+1}$. Therefore we have
	\begin{align*}
		R(d,k,t,\delta) & \leq t(t\delta^d)^k (k+1)^d \sum\limits_{r=1}^k \bC(k,r) \rbras*{(t\delta^d)^k (k+1)^d}^{\frac{k+1-r}{k}}.
	\end{align*}
	Since $t \rightarrow \infty$ and we assume \eqref{eq:phase:2} or \eqref{eq:phase:3} we have $(t\delta^d)^k(k+1)^d \rightarrow 0$, thus
	\begin{align*}
		R(d,k,t,\delta) & \leq \underbrace{t(t\delta^d)^k (k+1)^d}_{\rightarrow (k+1)!\theta} \sum\limits_{r=1}^k \bC(k,r) \rbras*{\underbrace{(t\delta^d)^k (k+1)^d}_{\rightarrow 0}}^{\frac{k+1-r}{k}} \rightarrow 0,
	\end{align*}
	which yields our claim.
	
	If the expectation tends to infinity \eqref{eq:phase:1} we directly obtain the claim from the lower variance bound.
\end{proof}

In the next step we check, that our $k$-simplex counting functional $F_k$ or its standardization $\widetilde{F_k}$ satisfy the condition $F \in \dom{D}$ resp. $\widetilde{F_k} \in \dom{D}$ imposed through the Malliavin-Stein method:

\begin{lemma}\label{lem:domD}
	For all $d \geq 1$ and all $k \geq 1$ there exists a constant $\bD^{\mathrm{dom}}(k) \in (0,\infty)$ only depending on $k$ such that
	\begin{align*}
		\int\limits_{W} \E*{(D_xF_k)^2} \leq \bD^{\mathrm{dom}}(k)\V*{F_k}.
	\end{align*}
	Furthermore, if the expectation tends to infinity \eqref{eq:phase:1}, then
	\begin{align*}
		\widetilde{F_k} \in \dom{D}.
	\end{align*}
	If the expectation converges to a positive constant \eqref{eq:phase:2}, then
	\begin{align*}
		F_k \in \dom{D}.
	\end{align*}
\end{lemma}

\begin{proof}
	We use the moment decomposition Lemma \ref{lem:decomp:2} on $D_xF$ to obtain
	\begin{align*}
		& \int\limits_{W} \E*{(D_xF_k)^2} \id \mu(x)\\
		= & \int\limits_{W} \sum\limits_{r=0}^{k} \bC(k,r) \smashoperator{\int\limits_{W^{2k-r}}} \leqdelta*{\ivar yk,x} \leqdelta*{\ivar yr, \ivar z{k-r}, x} \id \mu^k(\ivar yk) \id \mu^{k-r}(\ivar z{k-r}) \id \mu(x)\\
		= & \sum\limits_{r=0}^k \bC(k,r) \smashoperator{\int\limits_{W^{2k-r+1}}} \leqdelta*{\ivar yk,x} \leqdelta*{\ivar yr, \ivar z{k-r}, x} \id \mu^k(\ivar yk) \id \mu^{k-r}(\ivar z{k-r}) \id \mu(x)
	\end{align*}
	Since $x$ is used in both kernels we rename the variables $y_i \rightarrow y_{i+1}$ for all $i \in \cbras*{0,\ldots,k}$ and $x \rightarrow y_0$ and shift the index $r$ of the sum to obtain
	\begin{align*}
		& \int\limits_{W} \E*{(D_xF_k)^2} \id \mu(x)\\
		= & \sum\limits_{r=0}^k \bC(k,r) \smashoperator{\int\limits_{W^{2k-r+1}}} \leqdelta*{\ivar y{k+1}} \leqdelta*{\ivar y{r+1}, \ivar z{k-r}} \id \mu^{k+1}(\ivar y{k+1}) \id \mu^{k-r}(\ivar z{k-r})\\
		= & \sum\limits_{r=1}^{k+1} \bC(k,r-1) \smashoperator{\int\limits_{W^{2(k+1)-r}}} \leqdelta*{\ivar y{k+1}} \leqdelta*{\ivar y{r}, \ivar z{k+1-r}} \id \mu^{k+1}(\ivar y{k+1}) \id \mu^{k+1-r}(\ivar z{k+1-r})\\
		= & \sum\limits_{r=1}^{k+1} \tfrac{\bC(k,r-1)\bC(k+1,r)}{\bC(k+1,r)} \smashoperator{\int\limits_{W^{2(k+1)-r}}} \leqdelta*{\ivar y{k+1}} \leqdelta*{\ivar y{r}, \ivar z{k+1-r}} \id \mu^{k+1}(\ivar y{k+1}) \id \mu^{k+1-r}(\ivar z{k+1-r}).
	\end{align*}
	Defining 
	\begin{align*}
		\bD^{\mathrm{dom}}(k) = \max\limits_{r=1}^{k+1} \frac{\bC(k,r-1)}{\bC(k+1,r)}
	\end{align*}
	the first claim follows immediately. Using Lemma \ref{lem:variance:limit} completes the proof.
\end{proof}

We proceed to derive the upper bounds for $\gamma_1$, $\gamma_2$ and $\gamma_{3,N}$ resp. $\gamma_{3,P}$.

\begin{lemma}\label{lem:gamma:bound}
	For all $d \geq 1$ and all $k \geq 1$ the error terms in the Malliavin-Stein limit theorems are bounded by
	\begin{align*}
		\gamma_1(F_k) & \ll 4^d t(t\delta^d)^{k+\frac{3}{2}} (k+1)^{4d} \sum\limits_{q=0}^{12k-6} (t\delta^d)^{\frac{q}{4}},\\
		\gamma_2(F_k) & \ll 4^d t(t\delta^d)^{k+1} k^{4d} \sum\limits_{q=0}^{6k-6}(t\delta^d)^{\frac{q}{2}},\\
		\gamma_{3,P}(F_k) & \ll t(t\delta^d)^{k+\frac{1}{2}} (k+1)^{3d} \sum\limits_{q=0}^{4k-1} (t\delta^d)^{\frac{q}{2}}.
	\end{align*}
	The error terms for the standardized $k$-simplex counting functional are bounded by
	\begin{align*}
		\gamma_1(\widetilde{F_k}) & \ll \rbras*{\V*{F_k}}^{-2} 4^d t(t\delta^d)^{k+\frac{3}{2}} (k+1)^{4d} \sum\limits_{q=0}^{12k-6} (t\delta^d)^{\frac{q}{4}},\\
		\gamma_2(\widetilde{F_k}) & \ll \rbras*{\V*{F_k}}^{-2} 4^d t(t\delta^d)^{k+1} k^{4d} \sum\limits_{q=0}^{6k-6}(t\delta^d)^{\frac{q}{2}},\\
		\gamma_{3,N}(\widetilde{F_k}) & \ll \rbras*{\V*{F_k}}^{-\frac{3}{2}} \bC t(t\delta^d)^k (k+1)^{3d} \sum\limits_{q=0}^{2k} (t\delta^d)^q.
	\end{align*}
\end{lemma}

\begin{proof}
	Since the Malliavin difference operator is linear and invariant under addition of a constant we have $D_x\widetilde{F_k} = {D_xF_k}/{\sqrt{\V*{F_k}}}$. Therefore we can calculate the estimates for $F_k$ first and re-scale them later to obtain the results for $\widetilde{F_k}$.
	Fix $x_1, x_2, x_3 \in W$.
	Using \eqref{eq:derivative:bound:DxF} and \eqref{eq:derivative:bound:DxxF} we bound the integrand of $\gamma_1$ by
	\begin{align*}
		& \E*{(D_{x_1,x_3}^2F)^4}\E*{(D_{x_2,x_3}^2 F)^4} \E*{(D_{x_1}F)^4}\E*{(D_{x_2}F)^4}\\
		\ll ~ & \bD_4(k)^2 \bD_4'(k)^2 \leqdelta*{x_1,x_3}\leqdelta*{x_2,x_3}\sum\limits_{q_1,q_2=k-1}^{4(k-1)}\sum\limits_{q_3,q_4=k}^{4k} (t\delta^d)^{q_1+q_2+q_3+q_4}\\
		& \times \rbras*{(k+1)k}^{2d} \rbras*{\rbras*{\tfrac{q_1-(k-1)}{3}+1} \rbras*{\tfrac{q_2-(k-1)}{3}+1} \rbras*{\tfrac{q_3-k}{3}+1} \rbras*{\tfrac{q_4-k}{3}+1}}^{3d}\\
		\ll ~ & \leqdelta*{x_1,x_3}\leqdelta*{x_2,x_3} \sum\limits_{q_1,q_2=k-1}^{4(k-1)}\sum\limits_{q_3,q_4=k}^{4k} (t\delta^d)^{q_1+q_2+q_3+q_4} (k+1)^{16d},
	\end{align*}
	and the integrand of $\gamma_2$ by
	\begin{align*}
		& \E*{(D_{x_1,x_3}^2 F)^4}\E*{(D_{x_2,x_3}^2 F)^4}\\
		\ll ~ & \bD_4'(k)^2 \leqdelta*{x_1,x_3} \leqdelta*{x_2,x_3} \smashoperator{\sum\limits_{q_1,q_2=k-1}^{4(k-1)}} (t\delta^d)^{q_1+q_2} k^{2d} \rbras*{\rbras*{\tfrac{q_1-(k-1)}{3}+1} \rbras*{\tfrac{q_2-(k-1)}{3}+1}}^{3d}\\
		\ll ~ & \leqdelta*{x_1,x_3}\leqdelta*{x_2,x_3} \smashoperator{\sum\limits_{q_1,q_2=k-1}^{4(k-1)}}(t\delta^d)^{q_1+q_2} k^{8d}.
	\end{align*}
	Note that
	\begin{align*}
		\smashoperator{\int\limits_{W^3}} \leqdelta*{x_1,x_3}\leqdelta*{x_2x_3} \id\mu^3(x_1,x_2,x_3) \leq \Lambda_d(W) t^3\delta^{2d}4^d = t^3\delta^{2d}4^d.
	\end{align*}
	Since $\sqrt[n]{x}$ is subadditive for all $x > 0$ and $n > 1$, it follows that
	\begin{align*}
		\gamma_1(F_k) & \ll \smashoperator{\int\limits_{W^3}} \leqdelta*{x_1,x_3}\leqdelta*{x_2x_3} \id\mu^3(x_1,x_2,x_3) \smashoperator[l]{\sum\limits_{q_1,q_2=k-1}^{4(k-1)}}\smashoperator[r]{\sum\limits_{q_3,q_4=k}^{4k}} (t\delta^d)^{\frac{q_1+q_2+q_3+q_4}{4}} (k+1)^{4d}\\
		& \ll 4^d t^3\delta^{2d} (k+1)^{4d} \sum\limits_{q=4k-2}^{16k-8} (t\delta^d)^{\frac{q}{4}},
	\end{align*}
	and
	\begin{align*}
		\gamma_2(F_k) & \ll \smashoperator{\int\limits_{W^3}} \leqdelta*{x_1,x_3}\leqdelta*{x_2x_3} \id\mu^3(x_1,x_2,x_3) \sum\limits_{q_1,q_2=k-1}^{4(k-1)}(t\delta^d)^{\frac{q_1+q_2}{2}} k^{4d}\\
		& \ll 4^d t^3\delta^{2d} k^{4d} \sum\limits_{q=2k-2}^{8k-8}(t\delta^d)^{\frac{q}{2}},
	\end{align*}
	where we have simplified the sums using only the possible exponents of $(t\delta^d)$, since the number of terms with the corresponding exponent in the double sum does only depend on $k$.
	Shifting the indices of the sums such that the summations start at $q = 0$ establishes the $\gamma_1$ and $\gamma_2$ bounds for $F_k$ and re-scaling yields the corresponding bounds for $\widetilde{F_k}$.
	
	Fix $x \in W$. 
	Using \eqref{eq:derivative:bound:DxFDxF-1} and \eqref{eq:derivative:bound:DxF} we bound the integrand of $\gamma_{3,P}$ by
	\begin{align*}
		& \E*{\rbras*{(D_xF_k)(D_xF_k-1)}^2} \E*{(D_xF_k)^2}\\
		\ll ~ & \bD_4^*(k) \sum\limits_{q_1=k+1}^{4k} \sum\limits_{q_2=k}^{2k} (t\delta^d)^{q_1+q_2} (k+1)^{2d} \rbras*{\tfrac{q_1-k}{3}+1}^{3d}\rbras*{q_2-k+1}^{d}\\
		\ll ~ & \sum\limits_{q_1=k+1}^{4k} \sum\limits_{q_2=k}^{2k} (t\delta^d)^{q_1+q_2} (k+1)^{6d}.
	\end{align*}
	It follows that
	\begin{align*}
		\gamma_{3,P}(F_k) & \ll \int\limits_{W} 1 \id\mu(x) \sum\limits_{q_1=k+1}^{4k} \sum\limits_{q_2=k}^{2k} (t\delta^d)^{\frac{q_1+q_2}{2}} (k+1)^{3d}\\
		& \ll t (k+1)^{3d} \sum\limits_{q=2k+1}^{6k} (t\delta^d)^{\frac{q}{2}}.
	\end{align*}
	
	Using \eqref{eq:derivative:bound:DxF} together with $D_xF_k \geq 0$, we bound the integrand of $\gamma_{3,N}$ by
	\begin{align*}
		\Eabs*{(D_xF)^3} & \ll \bD_3(k) \sum\limits_{q=k}^{3k}(t\delta^d)^q (k+1)^d \rbras*{\tfrac{q-k}{2}+1}^{2d}\\
			& \ll \sum\limits_{q=k}^{3k} (t\delta^d)^q (k+1)^{3d}.
	\end{align*}
	It follows that
	\begin{align*}
		\gamma_{3,N}(F_k) & \ll \int\limits_{W} 1 \id\mu(x) \sum\limits_{q=k}^{3k} (t\delta^d)^q (k+1)^{3d}\\
		& \ll t(k+1)^{3d} \sum\limits_{q=k}^{3k} (t\delta^d)^q.
	\end{align*}
	Re-scaling this bound for $\widetilde{F_k}$ completes the proof.
\end{proof}

We are now in a position to use the Malliavin-Stein method, in particular Theorem \ref{thm:malliavin-gauss-limit} and Theorem \ref{thm:malliavin-poisson-limit}, to prove our main results, the central limit theorem and the Poisson limit theorem for the $k$-simplex counting functional.

\begin{proof}[Proof of Theorem \ref{thm:gaussian-limit}]
	Assume, that the expectation of $F_k$ tends to infinity \eqref{eq:phase:1}, i.e.
	\begin{align*}
		\lim\limits_{d \rightarrow \infty}\frac{1}{(k+1)!}t(t\delta^d)^k (k+1)^d = \infty.
	\end{align*}
	Lemma \ref{lem:variance:limit} yields $\V*{F_k} \rightarrow \infty$ and Lemma \ref{lem:domD} gives $\widetilde{F_k} \in \dom{D}$. This enables us, to use the Malliavin-Stein method to derive a bound on the Wasserstein-distance between $\widetilde{F_k}$ and a standard Gaussian distributed random variable $\cN(0,1)$ from Theorem \ref{thm:malliavin-gauss-limit}.
	We have to distinguish the following three cases based on the limit of $(t\delta^d)$:
	
	\textbf{Case 1:} If $(t\delta^d) \rightarrow 0$, the variance is bounded from below by
	\begin{align*}
		t (t\delta^d)^k(k+1)^d \ll \V*{F_k}.
	\end{align*}
	It follows that
	\begin{align*}
		\gamma_1(\widetilde{F_k}) & \ll 4^d (k+1)^{2d} t^{-1} (t\delta^d)^{\frac{3}{2}-k},\\
		\gamma_2(\widetilde{F_k}) & \ll 4^d k^{2d} \rbras*{\tfrac{k}{k+1}}^{2d} t^{-1} (t\delta^d)^{1-k},\\
		\gamma_{3,N}(\widetilde{F_k}) & \ll (k+1)^{\frac{3d}{2}} t^{-\frac{1}{2}} (t\delta^d)^{-\frac{k}{2}},
	\end{align*}
	and further
	\begin{align*}
		\sqrt{\gamma_1} + \sqrt{\gamma_2} + \gamma_{3,N}
		\ll
		\begin{cases}
			t^{-\frac{1}{2}}(t\delta^d)^{-\frac{k}{2}}2^d (k+1)^d, & k \leq 3,\\
			t^{-\frac{1}{2}}(t\delta^d)^{-\frac{k}{2}}(k+1)^{\frac{3d}{2}}, & \geq 3.
		\end{cases}
	\end{align*}
	Since $\E*{F_k} \ll t(t\delta^d)^k(k+1)^d$ it follows that 
	\begin{align*}
		t^{-\frac{1}{2}} (t \delta^d)^{-\frac{k}{2}}(k+1)^{-\frac{d}{2}} \ll \rbras*{\E*{F_k}}^{-\frac{1}{2}},
	\end{align*}
	which yields the desired result.
	
	\textbf{Case 2:} If $(t\delta^d) \rightarrow c \in (0,\infty)$, the variance is bounded from below by
	\begin{align*}
		t (k^2+2k)^d \ll \V*{F_k} 
	\end{align*}
	It follows that
	\begin{align*}
		\gamma_1(\widetilde{F_k}) & \ll 4^d t^{-1} \rbras*{1+\tfrac{1}{k^2+2k}}^{2d},\\
		\gamma_2(\widetilde{F_k}) & \ll 4^d t^{-1} \rbras*{\tfrac{k^2}{k^2+2k}}^{2d},\\
		\gamma_{3,N}(\widetilde{F_k}) & \ll t^{-\frac{1}{2}} \rbras*{1+\tfrac{1}{k^2+2k}}^{\frac{3d}{2}},
	\end{align*}
	and further
	\begin{align*}
		\sqrt{\gamma_1} + \sqrt{\gamma_2} + \gamma_{3,N}
		\ll
		2^d t^{-\frac{1}{2}} \rbras*{1+\tfrac{1}{k^2+2k}}^{d}.
	\end{align*}
	\textbf{Case 3:} If $(t\delta^d) \rightarrow \infty$, the variance is bounded from below by
	\begin{align*}
		t(t\delta^d)^{2k}(k^2+2k)^d \ll \V*{F_k}.
	\end{align*}
	It follows that
	\begin{align*}
		\gamma_1(\widetilde{F_k}) & \ll 4^d t^{-1}\rbras*{1+\tfrac{1}{k^2+2k}}^{2d},\\
		\gamma_2(\widetilde{F_k}) & \ll 4^d t^{-1} (t\delta^d)^{-2} \rbras*{\tfrac{k^2}{k^2+2k}}^{2d},\\
		\gamma_{3,N}(\widetilde{F_k}) & \ll t^{-\frac{1}{2}} \rbras*{1+\tfrac{1}{k^2+2k}}^{\frac{3d}{2}},
	\end{align*}
	and further
	\begin{align*}
		\sqrt{\gamma_1} + \sqrt{\gamma_2} + \gamma_{3,N}
		& \ll
		2^d t^{-\frac{1}{2}} \rbras*{1+\tfrac{1}{k^2+2k}}^d + t^{-\frac{1}{2}} \rbras*{1+\tfrac{1}{k^2+2k}}^{\frac{3d}{2}}\\
		& \ll
		2^d t^{-\frac{1}{2}} \rbras*{1+\tfrac{1}{k^2+2k}}^d.
	\end{align*}
	Using the Malliavin-Stein bound \eqref{eq:malliavin-gauss-limit} for $\dist_W$ completes the proof.
\end{proof}

\begin{proof}[Proof of Theorem \ref{thm:poisson-limit}]
	Assume, that the expectation of $F_k$ converges to a positive constant \eqref{eq:phase:2}, i.e.
	\begin{align*}
		\lim\limits_{d \rightarrow \infty}\frac{1}{(k+1)!} t(t\delta^d)^k (k+1)^d = \theta \in (0,\infty).
	\end{align*}
	Lemma \ref{lem:variance:limit} yields $\V*{F_k} \rightarrow \theta$ and Lemma \ref{lem:domD} gives $F_k \in \dom{D}$. This enables us, to use the Malliavin-Stein method to derive a bound on the total-variation distance between $F_k$ and a Poisson-distributed random variable $\cP(\theta)$ from Theorem \ref{thm:malliavin-poisson-limit}. 
	We note that our assumption implies $(t\delta^d) \rightarrow 0$ and $t(t\delta^d)^k(k+1)^d \rightarrow (k+1)!\theta$.
	
	Therefore 
	\begin{align*}
		\gamma_1(F_k) & \ll 4^d (t\delta^d)^{\frac{3}{2}}(k+1)^{3d},\\
		\gamma_2(F_k) & \ll 4^d (t\delta^d) k^{3d} \rbras*{\tfrac{k}{k+1}}^d,\\
		\gamma_{3,P}(F_k) & \ll (t\delta^d)^{\frac{1}{2}}(k+1)^{2d},
	\end{align*}
	and further
	\begin{align*}
		\sqrt{\gamma_1} + \sqrt{\gamma_2} + \gamma_{3,P}
		\ll 
		\begin{cases}
			(t\delta^d)^{\frac{1}{2}} 2^d (k+1)^{\frac{3d}{2}}, & k \leq 3,\\
			(t\delta^d)^{\frac{1}{2}} (k+1)^{2d}, & k \geq 3.
		\end{cases}
	\end{align*}
	It follows from the Malliavin-stein bound \eqref{eq:malliavin-poisson-limit} for $\dist_{TV}$, that
	\begin{align*}
		\dist_{TV}\rbras*{F_k,\cP(\theta)} & \ll \abs*{\E*{F_k} - \theta} + \abs*{\V*{F_k} - \theta} + 
		\begin{cases}
			(t\delta^d)^{\frac{1}{2}} 2^d (k+1)^{\frac{3d}{2}}, & k \leq 3,\\
			(t\delta^d)^{\frac{1}{2}} (k+1)^{2d}, & k \geq 3.
		\end{cases}\\
		& \ll \abs*{\E*{F_k} - \theta} + \abs*{\V*{F_k} - \theta} + 
		\begin{cases}
			t^{-\frac{1}{2k}} (k+1)^{\frac{d(3k-1)}{2k}} 2^{d}, & k \leq 3,\\
			t^{-\frac{1}{2k}} (k+1)^{\frac{d(4k-1)}{2k}}, & k \geq 3.
		\end{cases}
	\end{align*}
	which yields our claim.
\end{proof}

%\AtNextBibliography{\small}
\printbibliography

\listoffixmes

\end{document}